\documentclass{article}
\usepackage[utf8]{inputenc} 
\usepackage{amsmath,amssymb,amsfonts,bm}
\usepackage{multirow}
\usepackage{mathrsfs}
\usepackage{booktabs}
\usepackage{authblk}
\usepackage[ruled,vlined]{algorithm2e}
\usepackage{ragged2e} 
\usepackage{booktabs}
\usepackage{amsthm}
\usepackage{latexsym}
\usepackage{graphicx}
\usepackage{rotating}
\usepackage{subcaption}
\usepackage{color,varioref}
\usepackage{longtable}
\allowdisplaybreaks[4]
\usepackage[colorlinks,linkcolor=blue]{hyperref}

\usepackage{tablefootnote}

\newtheorem{theorem}{Theorem}[section]

\newtheorem{definition}{Definition}[section]
\newtheorem{lemma}{Lemma}[section]
\newtheorem{remark}{Remark}[section]

\usepackage{geometry}
\usepackage{enumitem}

\begin{document}
	
\title{Accelerated Proximal Dogleg Majorization for Sparse Regularized Quadratic Optimization Problem
}

\author[1,2]{Feifei Zhao}
\author[1]{Qingsong Wang}
\author[2]{Mingcai Ding}

\author[1]{Zheng Peng\thanks{Corresponding author. \\
Email: {\tt  zhaofeifei@shzu.edu.cn(Feifei Zhao),nothing2wang@hotmail.com(Qingsong Wang),dingmc@shzu.edu.cn(Mingcai Ding),pzheng@xtu.edu.cn(Zheng Peng).}}}

\affil[1]{School of Mathematics and Computational Science, Xiangtan University, Xiangtan 411105, Hunan Province, China}	
\affil[2]{College of Sciences, Shihezi University, Xiangyang Street, Shihezi, 832003, Xinjiang, People's Republic of China}	

\date{}	
\maketitle

\begin{abstract}
This paper addresses the problems of minimizing the sum of a quadratic function and a proximal-friendly nonconvex nonsmooth function. While the existing Proximal Dogleg Opportunistic Majorization (PDOM) algorithm for these problems offers computational efficiency by minimizing opportunistic majorization subproblems along mixed Newton directions and requiring only a single Hessian inversion, 
its convergence rate is limited due to the nonconvex nonsmooth regularization term, and its theoretical analysis is restricted to local convergence. To overcome these limitations, we firstly propose a novel algorithm named PDOM with extrapolation (PDOME). Its core innovations  lie in two key aspects: (1) the integration of an extrapolation strategy into the construction of the hybrid Newton direction, and (2) the enhancement of the line search mechanism. Furthermore, we establish the global convergence of the entire sequence generated by PDOME to a critical point and derive its convergence rate under the Kurdyka-Lojasiewicz (KL) property. Numerical experiments demonstrate that PDOME achieves faster convergence and tends to converge to a better local optimum compared to the original PDOM. 
\end{abstract}
	
	
\maketitle
\section{Introduction}
This paper investigates a class of nonconvex composite optimization problems, specifically considering objective functions formed by the sum of a convex quadratic function and a nonconvex nonsmooth function:
\begin{equation}
\label{eqz2}
\min_{{x}\in\mathbb{R}^n}Q({x}):=s({x})+r({x}),
\end{equation}
where ${x}$ is the decision variable, $s$ is a quadratic function with $\nabla^2 s({x})\succcurlyeq 0$\tablefootnote{When $\nabla^2{s}$   is positive semi-definite, an $\epsilon\boldsymbol{I}$  can be added into  $\nabla^2{s}$ where 
$\epsilon>0$ is small.} and
$r: \mathbb{R}^n \to \mathbb{R} \cup \{\infty\}$ is a nonconvex and nonsmooth function (which may represent regularization, constraints, or complex structures). We further assume that $Q$ is lower bounded, i.e., there exists a real number $h$ such that $\forall {x} \in \mathbb{R}^n$, $Q({x})\geq h$. Additionally, $r$ is proximal-friendly, which means that  proximal operator $\operatorname{prox}_{\eta r}(\cdot) := \arg\min_{\boldsymbol{x} \in \mathbb{R}^n} \{r({x}) + \frac{1}{2\eta} \|{x} - \cdot\|^2\}$ ( with step size $\eta > 0$ ) is easy to compute \cite{combettes2011proximal}.

The structured optimization problem defined  in \eqref{eqz2} arises in diverse signal processing and machine learning applications. A quintessential example is sparse signal recovery \cite{ghayem2017sparse}, which underpins techniques like channel estimation \cite{stankovic2018analysis}, audio processing \cite{sharp2008application}, and blind source separation \cite{zibulevsky2001blind}. To enforce sparsity, numerous regularization strategies
are employed, encompassing both convex
 methods, like the $l_{1}$ norm \cite{candes2008introduction} and nonconvex counterparts including the $l_{1/2}$ quasi-norm \cite{xu2012l_}, the $l_{0}$ pseudo-norm \cite{blumensath2009iterative}, and the capped-$l_{1}$ penalty \cite{zhang2010analysis}. Beyond sparse modeling, similar composite optimization formulations also appear in other key areas of machine learning, such as low-rank matrix completion \cite{huang2021robust} and robust principal component analysis (RPCA) \cite{candes2011robust}. 

\begin{table}[h]
\centering
\label{tab:algorithms}
\caption{Summarize whether the existing methods incorporate second-order information (SOI), adopt opportunistic majorization (OM), utilize extrapolation techniques, compute the inverse of the Hessian matrix once, and make assumption on $r$, as well as the convergence of subsequences\tablefootnote{``-" means~not~given.}. }
\renewcommand{\arraystretch}{1.1}

\resizebox{\textwidth}{!}{
\begin{tabular}{ccccccc}
\toprule
\textbf{Algorithm} & \textbf{Convexity of $r$} & \textbf{SOI} & \textbf{Extrapolation} & \textbf{OM Strategy} & \textbf{1 Hessian Inversion} & \textbf{Convergence} \\
\midrule
PG \cite{combettes2011proximal}                 & Yes & No & No & No & No & Global \\
PN \cite{lee2014proximal}                  & Yes & Yes & Yes & No & No & Local  \\
mAPG \cite{li2015accelerated}                & Yes & No & Yes & No & No & Global  \\
APGnc \cite{li2017convergence}                & Yes & No & Yes & No & No & Local \\
PANOC \cite{stella2017simple}            & No & Yes  & No & No & No & Global \\
PANOC+ \cite{de2022proximal}           & No &Yes  & Yes & No & No & Global \\
PDOM \cite{zhou2024proximal}             & No & Yes & No & Yes & Yes & Local \\
\textbf{sPDOME (ours)}             & No & Yes & Yes & Yes & Yes & - \\
\textbf{PDOME (ours)}            & No & Yes & Yes & Yes & Yes & Global \\
\bottomrule
\end{tabular}
}
\end{table}

Various algorithms exist for solving \eqref{eqz2}, among which the proximal gradient (PG) method \cite{combettes2011proximal} is the most widely adopted due to its simplicity and ease of implementation. PG method iteratively combines gradient descent on the smooth component with a proximal update for the nonsmooth term. However, in nonconvex settings, PG method suffers from slow convergence, exhibiting only sublinear rates of  $O(1/k)$ in the worst case (where $k$ is the iteration index) \cite{attouch2009convergence, attouch2013convergence}.
To accelerate convergence, techniques incorporating momentum (e.g., the accelerated proximal gradient method for nonconvex programming, APGnc \cite{li2017convergence}) or Nesterov extrapolation (e.g., the accelerated proximal gradient (APG) method \cite{li2015accelerated, beck2009fast}) have been developed. These methods employ adaptive mechanisms to dynamically select between standard PG updates and accelerated variants based on objective function values.
Although these accelerated variants offer improvements, their guarantees of faster convergence often require specific problem structures or assumptions \cite{li2015accelerated, attouch2009convergence, attouch2013convergence}. Consequently, the desired acceleration may fail to materialize in more general, unconstrained, or challenging nonconvex optimization scenarios.

Newton-type algorithms, exemplified by the proximal Newton method (PN) \cite{lee2014proximal}, have garnered significant recent attention for minimizing objectives comprising a twice-differentiable term and a proximal-friendly term. At its core, each iteration involves constructing a scaled proximal operator (SPO) derived from the Hessian of the differentiable term and solving the resulting subproblem. Crucially, when the objective is convex and the SPO can be efficiently solved, this approach achieves a superlinear asymptotic convergence rate. However, the proximal Newton method faces two fundamental limitations:(1) solving the SPO presents a significant computational challenge, and the development of efficient solvers has yet to be achieved; (2) its established fast convergence is confined to convex problems, providing no assurance for the nonconvex case.

Instead of directly addressing the computationally challenging SPO, quasi-Newton approaches \cite{stella2017simple, de2022proximal, patrinos2013proximal, stella2017forward, themelis2018forward} strategically minimize the forward-backward envelope (FBE), which shares the same local minimizers as the original objective function. These methods iteratively compute the FBE's gradient and update a quasi-Newton direction (using BFGS or L-BFGS \cite{liu1989limited}) based on this gradient, subsequently performing a line search along the minimization, it introduces two significant drawbacks: (1) the iterative updating of the Hessian approximation via BFGS/L-BFGS imposes a notable computational burden and memory overhead; (2) crucially, these approximations fail to exploit potential special structure present in the exact Hessian, potentially discarding valuable problem-specific information that could accelerate convergence.

Several algorithms leverage the core Newton-type principle of using approximate second-order information for fast convergence. Sepcially, Proximal Averaged Newton-type method for Optimal Control (PANOC) \cite{stella2017simple} innovatively integrates the forward-backward (FB) method and the FBE method to overcome their individual drawbacks, demonstrating effectiveness in nonlinear constrained optimal control and achieving superlinear convergence under mild assumptions. Building on PANOC, PANOC+ \cite{de2022proximal} addresses specific shortcomings (including those of its derivative, the PG algorithm) by introducing an adaptive step size rule tailored to the PG oracle. Validated through case studies, PANOC+ offers a complete convergence theory (handling local Lipschitz continuity) and robustness against suboptimal PG subproblem solutions. However, both PANOC and PANOC+, relying on an FBE penalty method within the Augmented Lagrangian Method (ALM) framework, can struggle with nonconvex nonsmooth constraints, often requiring approximations that degrade convergence speed. Complementing these, the PDOM algorithm \cite{zhou2024proximal} employs a majorization-minimization (MM) approach for problem \eqref{eqz2}, constructing a dogleg surrogate model that strategically combines gradient and Newton directions with proximal mapping. A key advantage of PDOM is that the quadratic nature of the smooth term keeps the Hessian constant, allowing its inverse to be precomputed and stored, thus eliminating the need for costly iterative matrix inversions typical of quasi-Newton methods. While PDOM provides theoretical guarantees for convergence to critical points and analyzes its local convergence rate, a significant limitation is that its convergence analysis, particularly the rate, is confined to the local domain,lacking established global convergence guarantees.

To address the limitations of the PDOM algorithm in \cite{zhou2024proximal}, we propose a simple PDOM algorithm called sPDOME, which incorporates an extrapolation parameter mechanism to significantly improve numerical performance. To further ensure global convergence, we develop the PDOME algorithm, which innovatively combines extrapolation techniques with an improved backtracking line search, thereby establishing a rigorous theoretical framework for convergence analysis.
The main contributions include:

\begin{itemize}
\item The core idea of  PDOME algorithm is based on the majorization-minimization (MM) framework and incorporates extrapolation acceleration techniques. This algorithm constructs a surrogate function along the dogleg path at the extrapolated point, integrating the gradient direction and Newton-type search direction: the gradient direction ensures the reliability of the sequence of iterates, while the Newton direction accelerates the local convergence rate. Further, optimizing the line search criterion helps determine a more optimal step size, thereby supporting subsequent convergence analysis. 

\item Through  theoretical analysis, we show that the limit points of the sequences generated by  PDOME  are critical points of the objective function. Then, by exploiting different cases of the Kurdyka-Lojasiewicz (KL) property of the objective function, we establish comprehensive convergence rate guarantees for PDOME, systematically characterizing its behavior across three distinct convergence regimes determined by the Lojasiewicz exponent. The theoretical analysis demonstrates that PDOME maintains computational efficiency comparable to conventional methods while requiring fewer proximal operator computations per iteration.

\item We conducted numerical experiments on several well-known nonconvex and nonsmooth problems. The results show that, compared with other benchmark algorithms, both the sPDOME algorithm and the PDOME algorithm can converge quickly.
\end{itemize}

The rest of this paper is structured as follows: Section \ref{sec:preliminaries} presents mathematical preliminaries and related preparatory work; Section \ref{sec:3} introduces the proposed sPDOME and PDOME algorithms, including a detailed analysis of PDOME algorithm convergence properties; Section \ref{sec:numerical_experiment} reports relevant experimental results; Section \ref{sec:5}  summarizes the paper and outlines future research directions.

\section{Preliminaries}\label{sec:preliminaries}
In this paper,  $\mathbb{R}^{n}$ 
is defined as the $n$ dimensional Euclidean space. The symbols $\cdot,\langle\cdot,\cdot\rangle$ and $T$  represent the standard product, inner product and transpose in the space $\mathbb{R}^{n}$. For an arbitrary vector ${x} \in \mathbb{R}^{n}$, the $\ell_{2}$ norm, the $\ell_{1}$ norm, and the $\ell_{0}$  pseudo-norm are defined as $\|{x}\| := \sqrt{{x}^{T}{x}}$,
$\|x\|_1 := \sum_{i=1}^{n} |x_i|$, and $\|x\|_0 := |\operatorname{supp}(x)|$ where $\operatorname{supp}(\cdot)$ counts the number of nonzero elements in $x$. Given a positive semidefinite matrix $\bm{\mathrm{M}} \in \mathbb{R}^{n \times n}$, the scaled norm of $x$ is defined as $\|x\|_{\bm{\mathrm{M}}} := \sqrt{x^T \bm{\mathrm{M}} x}$. Given a closed set $\Omega \subseteq \mathbb{R}^n$, $\operatorname{dist}(x, \Omega) := \inf \{\|y - x\|_2 : y \in \Omega\}$ calculates the distance between $x$ and $\Omega$.

\begin{definition} (Lower semicontinuous \cite{bertsekas2003convex}).
A function $Q:\mathbb{R}^n\to(-\infty,+\infty]$ is said to be proper if $\operatorname{dom} Q\neq\emptyset$, where $\operatorname{dom} Q=\{x\in\mathbb{R}^n:Q({x})<+\infty\}$, and lower semicontinuous at point $x_0$ if
\begin{align}
\lim_{{x}\to{x}_0}\inf Q({x})\geq Q\left({x}_0\right).\label{eq:2}
\end{align}
\end{definition}

\begin{definition} (Gradient Lipschitz continuity \cite{bertsekas2003convex}). Let $L_f\geq 0$. A         differentiable function $f:\mathbb{R}^n\to\mathbb{R}$ is said to have a Lipschitz continuity of the gradient if for all $x,{y}\in\operatorname{dom}f$ it holds that
\begin{align}
\|\nabla f({x})-\nabla f({y})\|\leq L_f\|{x}-{y}\|.\label{eq:3}
\end{align}
The corresponding Lipschitz constant is denoted as $L_f$.
\end{definition}
The value of $L_f$ for a twice differentiable function $f:$ $\mathbb{R}^n\to\mathbb{R}$ with a positive semi-definite Hessian matrix $\bm{\mathrm{M}}\in\mathbb{R}^{n\times n}$ is the largest eigenvalue of $\bm{\mathrm{M}}$, denoted by $\lambda_\mathrm{max}(\bm{\mathrm{M}}).$

\begin{definition} (Subdifferential     \cite{tyrrell1998variational}). Let $f:\mathbb{R}^{n}\to\mathbb{R}\cup\{+\infty\}$ be a proper and lower semicontinuous function. For a given $x\in\operatorname{dom}f$, the $Fr\acute{e}chet$ subdifferential of $f$ at $x$, written as
$\hat{\partial}f({x})$, is the set of all vectors ${v}\in\mathbb{R}^n$ which satisfy
$$\liminf_{y \to x, y \neq x} \frac{f(y) - f(x) - \langle v, y - x \rangle}{\|y - x\|} \geq 0.$$
For $x\notin\operatorname{dom}f$, we set
$\hat{\partial}f({x}):=\emptyset$.
The subdifferential (which is also called the limiting subdifferential) of $f$ at $x\in\mathbb{R}^{n}$, written as $\partial f(x)$, is defined by
\begin{align}
\partial f({x}):=\{{v}\in\mathbb{R}^n:\exists{x}^k\to{x}, f\left({x}^k\right)\to f({x}),{v}^k\in\hat{\partial}f\left({x}^k\right)\to{v}, k\to\infty\}.\label{eq:4}
\end{align}
As before $\partial f(x):=\emptyset$ for
$x\notin\mathrm{dom}f$, and its domain is 
$\mathrm{dom}~\partial f:=\{x\in \mathbb{R}^n:\partial f(x)\neq\emptyset\}$.
\end{definition}

\begin{definition}(KL property \cite{attouch2010proximal}). A proper closed function $r:\mathbb{R}^n\to\mathbb{R}\cup\{+\infty\}$ is said to have the KL property at $\hat{x}\in\operatorname{dom} \partial r$ if there exists $\eta\in(0,+\infty]$, a neighborhood $\mathcal{B}_\rho(\hat{x})\triangleq\{x:\|x-$
$\hat{x}\|<\rho\}$, and a continuous desingularizing concave function
$\psi:[0,\eta)\to[0,+\infty)$ with $\psi(0)=0$ such that\\
(i) $\psi$ is a continuously differentiable function with $\psi^{\prime}(x)>$
0, $\forall x\in ( 0, \eta )$,\\
(ii) for all $x\in\mathcal{B}_\rho(\hat{{x}})\cap\{{u}\in\mathbb{R}^n:r(\hat{{x}})<r({x})<r(\hat{{x}})+\eta\},$
it holds that
\begin{align}
\psi'(r({x})-r(\hat{{x}}))\operatorname{dist}(0,\partial r({x}))\geq1.\label{eq:5}
\end{align}
A proper closed function $r$ satisfying the KL property at all points in  $\operatorname{dom}\partial r$ is called a KL function.
\end{definition}

\begin{definition}(Eojasiewicz exponent \cite{yu2022kurdyka}). For a proper closed function $r$ satisfying the KL property at $\hat{x}\in\operatorname{dom}\hat{\partial}r$, if the desingularizing function $\psi$ can be chosen as $\psi(t)=$ $\frac{C}{1-2\theta}t^{2\theta-1}$ for some $C>0$ and $\theta\in[0,1)$, i.e., there exist $\rho>0$ and $\eta\in(0,+\infty]$ such that
\begin{align}
\operatorname{dist}(0,\partial r({x}))\geq C(r({x})-r(\hat{{x}}))^\theta, \label{eq:6}
\end{align}
where $x\in\mathcal{B}_\rho(\hat{x})$ and $r(\hat{x})<r({x})<r(\hat{{x}})+\eta$, then we say that $r$ has the KL property at $\hat{x}$ with an exponent of $\theta.$ We say that $r$ is a KL function with an exponent of $\theta$ if $r$ has the same exponent $\theta$ at any $\hat{x}\in\operatorname{dom}\partial r,$
where the desingularizing function $\psi$ can be chosen specifically as $\psi(t)=$ $\frac{C}{\theta}t^{\theta}$ with constants $C>0$ and $\theta\in(0,1]$.
\end{definition}
The KL property holds for a large family of functions used in optimization. For instance, all proper and closed semi-algebraic or subanalytic functions satisfy the KL property  with an associated Lojasiewicz exponent $\theta\in[0,1)$ \cite{attouch2010proximal}. The global convergence rate of PDOME is determined by the value of the Lojasiewicz exponent. In Subsection \ref{sec:convergence}, we provide the exponent value of the problem under test.

\begin{lemma} \label{lem:2.6} (Uniformized KL Property \cite{bolte2014proximal}). Let $\mathbb{W}$ be a compact set and $r : \mathbb{R}^n \rightarrow \mathbb{R} \cup \{ \infty \}$ be a proper and lower semicontinuous function. We assume that $r$ is constant on $\mathbb{W}$ and satisfies the KL property at each point of $\mathbb{W}$. Then, there exist $\epsilon > 0$, $\eta > 0$ and $\phi \in Y_\eta$ such that for all $\bar{z} \in \mathbb{W}$, one has
$$\phi'(h(z) - h(\bar{z}))\text{dist}(0, \partial h(z)) \geq 1,$$
for all $z \in \{ z \in \mathbb{R}^n \mid \text{dist}(z, \mathbb{W}) < \epsilon \} \cap \{ z \in \mathbb{R}^n \mid h(\bar{z}) < h(z) < h(\bar{z}) + \rho \}$.
\end{lemma}

\subsection{PG Method from the Majorization-minimization Angle}\label{sec:pg_mm}
In this subsection, we examine the proximal gradient method from the perspective of a majorization-minimization algorithm and identify certain limitations that lead to its slow convergence rate. The PG method is a well-established method for solving composite optimization problems of the form \eqref{eqz2}. At each iteration, the PG method performs a proximal line search in the direction of the negative gradient, employing a positive step size $\eta$. At a given point $y^k$, with $\zeta$ as the momentum coefficient and $c$ being a constant, PG   method solves the following surrogate function
\begin{align}
&\min_x~\underbrace{s\left({y}^k\right)+\left\langle\nabla s\left({y}^k\right),{x}-{y}^k\right\rangle+\frac{1}{2\eta}\left\|{x}-{y}^k\right\|^2}_{m_{pg}({x};{y}^k)}+r({x})\nonumber \\
&=\frac{1}{2\eta}\left\|{x}-\left({y}^k-\eta\nabla s\left({y}^k\right)\right)\right\|^2+r({x})+c,\label{eq:7}
\end{align}
where 
\begin{align}
&y^{k}=x^{k}+\zeta\left(x^{k}-x^{k-1}\right).
\label{eq:8}
\end{align}
Assuming that $r(\cdot)$ is proximable, the solution of \eqref{eq:7} can be obtained efficiently from a computational standpoint.
The iterates generated by the algorithm yield a nonincreasing sequence of objective  function's values. This property is ensured by the fact that the surrogate function $m_{pg}({x};{y}^{k})$ serves as an upper bound for $s({x})$. Specifically, $m_{pg}({x};{y}^{k})\geq s({x})$ holds for all ${x}\in\operatorname{dom}Q$, provided that $\eta<1/L_{s}$.
\begin{align}
Q({x}^{k+1})&=s({x}^{k+1})+r({x}^{k+1})\nonumber\\
&{\leq} m_{pg}({x}^{k+1};{y}^{k})+r({x}^{k+1})\nonumber\\
&{\leq}m_{pg}({x}^k;{y}^{k})+r({x}^k),\label{eq:9}
\end{align}
where the first inequality follows from the majorization step, and the second arises from the proximal operator's property.  A proximal gradient method with extrapolation and line search (PGels) is proposed in \cite{yang2024proximal} to address composite optimization problems that are potentially nonconvex, nonsmooth, and non-Lipschitz. By constructing an auxiliary function, the global subsequential convergence of PGels is proved. With appropriate parameter selection, PGels can be simplified to PG and PGe. The convergence guarantee of $Q\left(x^{k}\right)$ can be treated in the same way as in \cite{yang2024proximal}.

However, using the negative gradient direction often results in slow convergence, especially for nonconvex functions \cite{zhang2010analysis, attouch2013convergence, chen2015fast}. Furthermore, in cases where the Hessian matrix exhibits a large condition number, gradient-based methods become inefficient because of excessively slow convergence, as noted in \cite[Chapter 9]{boyd2004convex}.

\subsection{Hybrid Direction and Opportunistic Majorization}
To address the limitations identified in Subsection \ref{sec:pg_mm}, we propose the sPDOME and PDOME algorithm. These approaches integrates a dogleg search strategy, drawing inspiration from trust region methods. At each iteration, the algorithm constructs and minimizes an opportunistically majorized surrogate function along the dogleg path, replacing conventional gradient-based updates.

Given $\mu\in(0,2]$, the dogleg path is denoted as
\begin{align}\label{eq:10}  
d(\mu) :=  
\arraycolsep=1.0pt\def\arraystretch{1.5}  
\left\{
\begin{array}{ll}  
  d_{\eta} & \mu \in (0,1], \\  
  d_{\eta} + (\mu - 1)(d_N - d_\eta)\quad  & \mu \in (1,2],  
\end{array}  
\right.
\end{align} 
where $g:=\nabla s(y)$, $d_\eta:=-\eta\nabla s(y)$, and $d_N:=-(\nabla^2s(y))^{-1}\nabla s(y)$, $\eta$ is the fixed step size of the gradient direction, with $\eta\in(0,1/L_s)$, and $d_N$ denotes Newton point. The gradient direction is essential for ensuring convergence to a critical point, because at $x^\mathrm{cri}$, the first-order optimality condition of \eqref{eqz2} implies $0\in\partial Q(x^\text{cri})=\nabla s(x^\text{cri})+\partial r(x^\text{cri})$, for which the gradient direction is necessary. The construction of our path diverges fundamentally from the approach in \cite[Chapter~4]{nocedal2006numerical}, specifically by excluding the scaling factor $\mu$ in the initial segment. The adoption of this modification is warranted as $d_{\eta}$ intrinsically functions as a descent direction, all while maintaining freedom from trust-region restrictions. However, the path continues to ensure descent with respect to the quadratic term.

Given a positive definite matrix $\bm{\mathrm{M}}\succ0$ with bounded eigenvalues, it holds that
\begin{align}
\lambda_{\min}\|g\|^2 \leq \|g\|_{\bm{\mathrm{M}}}^2 = g^{T }\bm{\mathrm{M}} g \leq \lambda_{\max}\|g\|^2, \label{eq:11}
\end{align}
where $\lambda_\mathrm{max}$ and $\lambda_\mathrm{min}$ denote the largest and the smallest eigenvalue of $\bm{\mathrm{M}}$, respectively. For the sake of subsequent analysis, we perform a coordinate transformation that shifts the origin to the point $(y^k, s(y^k))$. Under this new coordinate system, the smooth part of the objective function and its corresponding surrogate are rewritten as
\begin{align}
s({x}):=\langle{g},{x}\rangle+\frac{1}{2}\|{x}\|_{{\bm{\mathrm{M}}}}^2,\quad m_\mu({x};y^{k}):=\langle{g}_\mu,{x}\rangle+\frac{1}{2\eta_\mu}\|{x}\|^2.\label{eq:12}
\end{align}

\section{Main Results }\label{sec:3}
In this section, we introduce the PDOME algorithm for solving problem \eqref{eqz2} which may involve nonconvexity and nonsmoothness, aiming to address the limitations outlined in Subsection \ref{sec:pg_mm}. Building on the existing technical framework, we further propose a novel extrapolation technique to enhance the algorithm's performance. Specifically, on the basis of extrapolation, we integrate the dogleg search strategy originally developed in the trust region methodology into the PDOME algorithm, replacing the conventional gradient-based descent direction with a hybrid search direction. This integration enables the algorithm to utilize more sophisticated directional information, thereby potentially accelerating convergence and improving robustness.

\begin{lemma} \label{lem:3.1}The inequality in the following equation,
\begin{align}
\langle d(\mu), \nabla s(d(\mu)) \rangle \leq 0,\label{eq:13}
\end{align}
is satisfied when $\eta = \frac{1}{\lambda_{\max}}$, where $\lambda_\mathrm{max}$  denote the largest eigenvalue of $\bm{\mathrm{M}}$. For any other $\eta \in (0, \frac{1}{\lambda_{\max}})$, the strict inequality holds.
\end{lemma}
\begin{proof} We begin with the trivial case where $\mu \in (0,1]$,
\begin{align}
\langle d(\mu),\nabla s(d(\mu))\rangle 
&=-\eta g(-\eta \bm{\mathrm{M}}g+g)\nonumber\\
&=\eta^2\left(\|g\|_{\bm{\mathrm{M}}}^2-\frac{1}{\eta}\|g\|^2\right)\nonumber\\
&\leq\eta^2\left(\lambda_{\max}\|g\|^2-\frac{1}{\eta}\|g\|^2\right)\nonumber\\
&\leq0.\nonumber
\end{align}
Then for the remaining range of $\mu \in (1,2]$, we have
\begin{align}
\langle d(\mu),\nabla s(d(\mu))\rangle 
&=-\eta\|g\|^2+(\mu-1)\left(\eta\|g\|^2-\|g\|^2_{{\bm{\mathrm{M}}}^{-1}}\right) \nonumber\\
&< (\mu-1)\left(\eta\|g\|^2-\|g\|^2_{{\bm{\mathrm{M}}}^{-1}}\right) \nonumber\\
&\leq 0.
\label{eq:14}
\end{align}
The equality of \eqref{eq:14} only holds when $\eta$ is exactly $1/\lambda_{\max}$.
\end{proof}
The effectiveness of an MM algorithm fundamentally depends on the construction of a surrogate function that acts as a tight upper bound for the original objective function. In PDOME, the local surrogate function $m_{\mu}$ is the projection of $m_{pg}$ onto the path direction, that is
\begin{align}
m_{\mu}({x};{y}^{k}):&=s({y}^{k})+\left\langle{g}_{\mu},{x}-{y}^{k}\right\rangle+\frac{1}{2\eta_{\mu}}\left\|{x}-{y}^{k}\right\|^{2}\nonumber\\
&=s({y}^k)+\frac{1}{2\eta_\mu}\left\|{x}-\left({y}^k+{d}(\mu)\right)\right\|^2-\frac{\eta_\mu}{2}\left\|{g}_\mu\right\|^2,
\label{eq:15}
\end{align}
where $g_\mu=\frac{\langle\nabla s\left(y^k\right),d(\mu)\rangle}{\|d(\mu)\|^2}d(\mu)$ and $\eta_\mu=-\frac{\|d(\mu)\|^2}{\langle\nabla s(y^k),d(\mu)\rangle}$, for $\mu\in(0,2]$.
The step size $\eta_\mu$ is allowed to surpass $\eta$.

\begin{lemma}\label{lem:3.2} $\eta_\mu$ is an increasing function of $\mu\in[0,2]$.
\end{lemma}
\begin{proof} The step size $\eta_\mu$ is considered an increasing function of $\mu$ within the interval $[0,2]$ provided that $\frac{d}{d\mu}\eta_\mu \geq 0$. The positive gradient can be proved with simple algebra.
\end{proof}

\begin{lemma}\label{lem:3.3}The sequence
$\left\{\eta_{\mu^k}\right\}_{k\in\mathbb{N}}$ is~bounded.
\end{lemma}
\begin{proof}Based on the definition of $\eta_{\mu}$, it holds that \\
\begin{align}
\eta_{\mu^k} &= -\frac{\|d(\mu^k)\|^2}{\langle \nabla s(y^k), d(\mu^k) \rangle} \nonumber\\
&= -\frac{\|d(\mu^k)\|^2}{(\mu^k - 2)\eta \| \nabla s(y^k) \|^2 +(1 - \mu^k){\left\|\nabla s({y}^k)\right\|_{{\bm{\mathrm{M}}}^{-1}}^2}}.
\label{eq:16}
\end{align}
Since the eigenvalues of ${\bm{\mathrm{M}}}$ are bounded, both the numerator and denominator are constrained by finite values. Consequently, the sequence
$\{\eta_{\mu^k}\}_{k\in\mathbb{N}}$ is established as bounded.
\end{proof}

In each iteration, the update rule is
\begin{align}
x^{k+1} &= \operatorname{prox}_{\eta_{\mu_k} r}\left( y^k + d(\mu^k) \right)\nonumber\\
&=\arg \min_{x} r(x) + \frac{1}{2\eta_{\mu_k}} \|x - (y^k + d(\mu^k))\|^2.
\label{eq:17}
\end{align}
It is not guaranteed that the new iterate $x^{k+1}$ results in a lower objective function value, since there is no assurance that $m_\mu(x; y^k)$ majorizes $s(x)$ for all choices of $\mu$. In what follows, we examine the MM condition under different ranges of $\mu$. We begin with the case where $\mu \in (0, 1]$, in which $m_\mu$ reduces to $m_{pg}$ with $g_\mu$ becoming $g$ and $\eta_\mu$ simplifying to $\eta$. According to \eqref{eq:7}, the surrogate function $m_\mu$ provides a uniform upper bound on $s$, meaning that
\[m_\mu({x};y^{k})\geq s({x}),\quad\forall{x}\in\operatorname{dom}f.
\]

Next, we consider the case where $\mu\in (1,2]$. In this range, the surrogate function $m_{\mu}$ only majorizes $s$ along the specific line segment that connects the current iterate and the path point. Unlike the classical MM principle cited in \cite{sun2016majorization,qiu2016prime}, where the surrogate is required to upper bound the objective function globally or over the entire domain, $m_{\mu}$ does not necessarily remain above function  $s$ everywhere. We refer to this more flexible condition as OM.

\begin{theorem}\label{thm:3.1} For any given $\mu\in(1,2]$, consider the line
connecting 0 and $d(\mu)$ which is given by
$$\mathcal{X}_\mu:=\left\{{x}(\beta):=\beta{d}(\mu):\:\forall\beta\in\mathbb{R}\right\}.$$
Define $\bar{s}(\beta):=s({x}(\beta))$ and $\bar{m}_{\mu}(\beta):=m_{\mu}({x}(\beta);y^{k})$. It holds that $\bar{s}(\beta)\leq\bar{m}_{\mu}(\beta)$ for all $\beta\in\mathbb{R}$, or equivalently, $s(x)\leq m_\mu({x};y^{k})$ for all $x\in\mathcal{X}_{\mu}$.
\end{theorem}
\begin{proof}
The proof of the theorem can be established by considering Lemma \ref{lem:3.4}, Lemma \ref{lem:3.5}, and Lemma \ref{lem:3.6}.
\end{proof}

\begin{lemma}\label{lem:3.4} Given $\bar{s}$ and $\bar{m}_{\mu }$ defined in Theorem \ref{thm:3.1}, it holds that $\bar{s}(0)=\bar{m}_{\mu }(0)=0$ and ${\bar{s}'\left (0\right )}={\bar{m }_{\mu } }'\left (0\right ) <0$.
\end{lemma}
\begin{proof} It is easy to show that $\bar{s}(0)=s(x(0))=0$, and $\bar{m}_\mu(0)=m_{\mu}(x(0);y^{k})=0$. Then we prove the negative gradient. It holds that
$$\bar{s}^{\prime}(0)=\beta{d}(\mu)^{\mathsf{T}}{{\bm{\mathrm{M}}}}{d}(\mu)+{g}^{\mathsf{T}}{d}(\mu)|_{\beta=0}=\langle{g},{d}(\mu)\rangle, $$
$$\bar{m}_{\mu}^{\prime}(0)={g}_{\mu}^{\mathsf{T}}{d}(\mu)+\frac{1}{\eta_{\mu}}\beta\|{d}(\mu)\|^{2}|_{\beta=0}=\langle{g},{d}(\mu)\rangle.$$
From Lemma \ref{lem:3.1}, we prove that
$\bar{s}^{\prime}(0)=\bar{m}_{\mu}^{\prime}(0)<0.$
\end{proof}

\begin{lemma}\label{lem:3.5}Let $m(x;y^{k})$ and $s(x)$ be univariate strictly convex quadratic functions. Suppose that\\
(i) $m(0)=s(0)$ and $m'(0) = s'(0) \neq 0$; \\
(ii) $x_{m}^{\star}= \eta x_{s}^{\star}$  for some  $\eta \in (0,1)$, where
$x_{m}^{\star} := \operatorname*{arg\,min}_x m(x)~and~x_{s}^{\star} := \operatorname*{arg\,min}_x s(x),$

then, it holds that
\begin{itemize}
    \item $\begin{aligned}[t]
        \left|\frac{x_m^\#}{m^{\prime}(0)}\right| < \left|\frac{x_s^\#}{s^{\prime}(0)}\right|;
    \end{aligned}$
    
    \item $\begin{aligned}[t]
         m(x;y^{k}) \geq s(x) \text{ for all } x, \text{ and the equality holds if and only if } x=0.
    \end{aligned}$
\end{itemize}

\end{lemma}
\begin{proof} As both $m(x;y)$ and $s(x)$ are quadratic, one can write them as $s(x)=s(0)+s'(0)x+\frac{1}{2\eta_s}x^2$ and $m(x;y^{k})=s(0)+s'(0)x+\frac{1}{2\eta_m}x^2$. It is clear that $x_q^\#= -\eta_s s'(0), x_m^\#= -\eta_m s'(0)$. From the assumption that
$x_m^\#=\arg\min_xm(x;y^{k})=\eta\arg\min_xs(x)=\eta x_s^\#$,
it holds that
$\left|\frac{x_m^\#}{m'(0)}\right|=\eta_m = \eta\eta_f < \eta_f = \left|\frac{x_s^\#}{s'(0)}\right|$,
or equivalently, $\frac{1}{\eta_m} > \frac{1}{\eta_f}$. Therefore, $m(x;y^{k}) \ge s(x)$ where the equality holds if and only if $x=0$. Both claims in the lemma are therefore proved.
Based on Lemma \ref{lem:3.4} and Lemma \ref{lem:3.5}, Theorem \ref{thm:3.1} can be proved by showing the lemma below.
\end{proof}

\begin{lemma}\label{lem:3.6} Considering $\bar{s}(\beta)$ and $\bar{m}(\beta)$ defined in Theorem \ref{thm:3.1}, it holds that $1=\arg\min_{\beta}\bar{m}_{\alpha}(\beta)\leq\arg\min_{\beta}\bar{s}(\beta).$
\end{lemma}
\begin{proof}
 It is established that
$$\bar{m}_\mu'(1)=\left.{g}_\mu^\mathsf{T}{d}(\mu)+\frac{1}{\eta_\mu}\beta\|{d}(\mu)\|^2\right|_{\beta=1}=0.$$
The claim that $1 = \arg\min_\beta \bar{m}_\mu(\beta)$ is therefore proved. We now show that $\bar{s}'(1) \leq 0$. It is clear that $$\bar{s}'(1) = \left.\beta d(\mu)^T {\bm{\mathrm{M}}}d(\mu) + g^T d(\mu)\right|_{\beta=1} \leq 0,$$
where the last inequality comes from  Lemma \ref{lem:3.1}. Combining this with Lemma \ref{lem:3.4} that $\bar{Q}^\prime(0)=\bar{m}_\mu^\prime(0)<0$, it can be concluded that $1=\arg\min_{\beta}\bar{m}_{\mu}(\beta)\leq\arg\min_{\beta}\bar{s}(\beta).$ This completes the proof.
\end{proof}
Proximal line search-based algorithms implicitly employ the principle of OM, although this principle is typically not explicitly recognized or articulated \cite{bonettini2017convergence}. Herein, we explicitly define the OM concept and integrate it into a Newton-type optimization framework.

\subsection{Algorithm Development}
Theorem \ref{thm:3.1} indicates that, by using the non-trivial surrogate function \eqref{eq:15}, the majorization condition is satisfied as long as the new iterate lies along the specified line. This ensures that the sequence $\left\{Q\left(x^{k}\right)\right\}_{k\in\mathbb{N}}$ is monotonically decreasing. To accelerate the convergence rate of the proximal gradient method, Ochs et al. \cite{ochs2014ipiano} introduced an inertial mechanism commonly referred to as extrapolation \cite{yang2024proximal,wen2017linear} into the proximal update framework. Based on the improved algorithm presented in \cite{zhou2024proximal}, we thus develop  the sPDOME algorithm, designed for solving problem \eqref{eqz2}. The corresponding iterative scheme is formulated as follows
\begin{align}
m_{\gamma,\mu}({x};{y}^{k}):&=s({y}^{k})+\left\langle{g}_{\mu},{x}-{y}^{k}\right\rangle+\frac{1}{2\gamma\eta_{\mu}}\left\|{x}-{y}^{k}\right\|^{2}\nonumber \\
&=s({y}^{k})+\frac{1}{2\gamma\eta_{\mu}}\left\|{x}-\left({y}^{k}+{d}_{\gamma}(\mu)\right)\right\|^{2}-\frac{\gamma\eta_{\mu}}{2}\left\|{g}_{\mu}\right\|^{2},
\label{eq:18}
\end{align}
where $d_\gamma(\mu)=\gamma{d}(\mu)$, and $\gamma\in(0,1)$ is a constant and typically set close to 1 in numerical experiments. The new iterate is
\begin{align}
x^{k+1}:&=\operatorname{prox}_{\gamma \eta_{\mu^k}r}(y^k+d_\gamma(\mu^k))\nonumber \\
&=\arg\min_{{x}}r({x})+\frac{1}{2\gamma\eta_{\mu^k}}\|x-(y^k+d_\gamma(\mu^k))\|^2,
\label{eq:19}
\end{align}
which remains easy to solve given the assumption that the standard proximal operator is computationally simple. The overall algorithm is summarized below. To find the largest $\mu^k$, Line 5 uses a strategy similar to that presented in \cite{stella2017simple}.

\begin{algorithm}[H]
\caption{sPDOME: simple Proximal Dogleg Opportunistic Majorization with Extrapolation}\label{Algorithm1}
\SetAlgoLined 
\DontPrintSemicolon 
\SetNlSty{}{}{} 

\KwIn{$x^0 \in \mathbb{R}^n, \bm{\mathrm{M}}^{-1} \in \mathbb{R}^{n \times n}, \eta \in (0, 1/L_s), \gamma \in (0, 1), \zeta \in (0, 1), k=0.$} 
\KwOut{$x^k.$} 
\Repeat {\mbox{stopping conditions are satisfied.}}
{ 
    $y^k = x^k + \zeta(x^k - x^{k-1}).$\; 
    Compute $x^{k+1}$ using $x^{k+1} = \text{prox}_{\gamma\eta_{\mu^k} r}(y^k + d_\gamma(\mu^k))$ for the largest $\mu^k \in \{1 + (1/2)^i \mid i \in \mathbb{N}\}$ such that $m_{\mu^k}(x^{k+1}; y^k) \geq s(x^{k+1}).$\;
    Compute $v^{k+1}$ using $v^{k+1} = \text{prox}_{\eta r}(y^k - \eta \nabla s(y^k)),$\;
    \If{$Q(x^{k+1}) > Q(v^{k+1})$}{
        set $x^{k+1} = v^{k+1}.$\;
    }
    
    \SetNlSty{textbf}{}{.} 
    \setcounter{AlgoLine}{6} 
    
    $k \leftarrow k + 1.$\; 
} 
\end{algorithm}

\begin{remark}
If $\zeta=0$, the sPDOME Algorithm \ref{Algorithm1} reduces to $PDOM$ Algorithm \cite{zhou2024proximal}.
\end{remark}

Building upon Algorithm \ref{Algorithm1}, we made some minor adjustments to the range of the inertial coefficient $\zeta$ 
and improved the line search step for
$\mu^{k}$ to be found, such as $\langle-\nabla s\left({y}^{k}\right)+\frac{\langle \nabla s\left({y}^{k}\right),{d}(\mu^{k})\rangle}{\|{d}(\mu^{k})\|^{2}}{d}(\mu^{k}),x^{k}-y^{k}\rangle\leq0$, thus obtaining the PDOME method, see  Algorithm \ref{Algorithm2} for details.

\begin{algorithm}[H]
\caption{PDOME: Proximal Dogleg Opportunistic Majorization  with extrapolation}\label{Algorithm2}
\SetAlgoLined 
\DontPrintSemicolon 
\SetNlSty{}{}{} 

\KwIn{$x^0 \in \mathbb{R}^n, \bm{\mathrm{M}}^{-1} \in \mathbb{R}^{n \times n}, \eta \in (0, 1/L_s), \gamma \in (0, 1), \zeta\in(0,\frac{1-\gamma}{2-\gamma}), k=0.$} 
\KwOut{$x^k.$} 
\Repeat{ stopping conditions are satisfied.}{ 
    $y^k = x^k + \zeta(x^k - x^{k-1}).$\; 
    Compute $x^{k+1}$ using $x^{k+1} = \text{prox}_{\gamma\eta_{\mu^k} r}(y^k + d_\gamma(\mu^k))$ for the largest $\mu^k \in \{1 + (1/2)^i \mid i \in \mathbb{N}\}$ such that $m_{\mu^k}(x^{k+1}; y^k) \geq s(x^{k+1})$
    and $\left\langle-\nabla s\left({y}^{k}\right)+\frac{\left\langle \nabla s\left({y}^{k}\right),{d}(\mu^{k})\right\rangle}{\|{d}(\mu^{k})\|^{2}}{d}(\mu^{k}),x^{k}-y^{k}\right\rangle\leq0.$
    \;
    Compute $v^{k+1}$ using $v^{k+1} = \text{prox}_{\eta r}(y^k - \eta \nabla s(y^k)).$\;
    \If{$Q(x^{k+1}) > Q(v^{k+1})$}{
        set $x^{k+1} = v^{k+1}.$\;
    }
    
    \SetNlSty{textbf}{}{.} 
    \setcounter{AlgoLine}{6} 
    
    $k \leftarrow k + 1.$\; 
} 
\end{algorithm}
The PDOME algorithm is designed to terminate upon approaching a critical point $x^{\ast}$, at which the condition $0\in\partial Q(x^{\star})$ is satisfied. Based on the optimality condition associated with the proxima operator in \eqref{eq:19}, the following relation holds:
\begin{align}
& 0\in\frac{1}{\gamma\eta_{\mu^{k}}}\left({x}^{k+1}-{y}^{k}-{d}_{\gamma}(\mu^{k})\right)+\partial r({x}^{k+1}), &
\label{eq:20}
\end{align}
this implies
\begin{align}
\partial Q({x}^{k+1})&=\nabla s({x}^{k+1})+\partial r({x}^{k+1}) \nonumber \\
& \ni\nabla s({x}^{k+1})+\frac{1}{\gamma\eta_{\mu^k}}\left({y}^k+{d}_\gamma(\mu^k)-{x}^{k+1}\right) \nonumber \\
& =\left(\nabla s({x}^{k+1})-{g}_{\mu^k}\right)-\frac{1}{\gamma\eta_{\mu^k}}\left({x}^{k+1}-{y}^k\right)\nonumber \\
&=\left(\nabla s({x}^{k+1})-{g}_{\mu^k}\right)-\frac{1}{\gamma\eta_{\mu^k}}\left({x}^{k+1}-{x}^k-\zeta({x}^k-{x}^{k-1})\right)\nonumber \\
&=\left(\nabla s({x}^{k+1})-{g}_{\mu^k}\right)-\frac{1}{\gamma\eta_{\mu^k}}({x}^{k+1}-{x}^{k})+\frac{\zeta}{\gamma\eta_{\mu^k}}({x}^k-{x}^{k-1}).  \label{eq:21}
\end{align}
PDOME terminates when $\|\partial Q(\boldsymbol{x}^{k+1})\|$  is sufficiently small:
\begin{equation}
\begin{split}
\|\partial Q(x^{k+1})\|& \leq \sqrt{n}\epsilon^{\mathrm{abs}} + \epsilon^{\mathrm{rel}} \max\Big\{
\|\nabla s(x^{k+1})\|, \ \|g_{\mu^k}\|, \frac{1}{\gamma\eta_{\mu^k}}\|x^{k+1}\|,
\frac{\zeta+1}{\gamma\eta_{\mu^k}}\|x^k\|, \ \frac{\zeta}{\gamma\eta_{\mu^k}}\|x^{k-1}\|
\Big\},
\label{eq:22}
\end{split}
\end{equation}
where $n$ is the dimension of $x$, $\epsilon^{abs}>0$ and $\epsilon^{rel}>0$ are two small positive constants (motivated~by~
\cite[Section~3.3]{boyd2011distributed}).

This stopping criterion is different from directly using $\|x^{k+1}-x^k \|$, commonly adopted for proximal algorithms \cite{lee2014proximal}. The relationship between these two different stopping criteria can be roughly quantified by the triangle inequality
\begin{align}
\left\|\partial Q({x}^{k+1})\right\|&\leq\left\|\nabla s({x}^{k+1})-{g}_{\mu^k}\right\|+\frac{1}{\gamma\eta_{\mu^k}}\left\|{x}^{k+1}-{x}^k\right\|
+\frac{\zeta}{\gamma\eta_{\mu^k}}\|{x}^k-{x}^{k-1}\|.
\label{eq:23}
\end{align}
As $1/\eta_{\mu^k}$ in \eqref{eq:23} can be very large, small value of $\|x^{k+1}-x^k\|$ does not necessarily imply getting close to a critical point. Now we formally present the PDOME in  Algorithm \ref{Algorithm2}. To track the largest $\mu^k$, we employ a similarly straightforward strategy as presented in \cite{stella2017simple}.

\subsection{The Convergence and Convergence Rate Analysis}\label{sec:convergence}
In this subsection, we analyze the convergence behavior of the sequence generated by the  Algorithm \ref{Algorithm2}, showing that it converges to a critical point of $H_\delta(x_{k})$. Additionally, under the KL condition, we derive the global convergence rate. We begin by establishing the monotonic decrease of the objective function values throughout the iterations.
\begin{theorem}
The sequence $\left\{Q(x^k)\right\}_{k\in\mathbb{N}}$ generated by
Algorithm \ref{Algorithm2} satisfies the following inequality.
\end{theorem}
\begin{proof}
It holds that
\begin{align}
Q({x}^{k+1})&= r({x}^{k+1}) + s({x}^{k+1}) \nonumber \\
&{\leq} r({x}^{k+1}) + m_{\mu}({x}^{k+1}; {y}^{k}) \nonumber \\
&{\leq} r({x}^{k+1}) + m_{\gamma,\mu}({x}^{k+1}; {y}^k) \nonumber \\
&{\leq}r({x}^k)+m_{\gamma,\mu}({x}^k;{y}^k),
\label{eq:24}
\end{align}
where the first inequality follows from the backtracking rule, the second inequality holds by virtue of $0<\gamma < 1$, and the third inequality from the proximal operator.
\end{proof}

\begin{lemma}\label{lem:3.7}
Suppose that $\left\{x^k\right\}_{k\in\mathbb{N}}$ is a sequence generated by Algorithm \ref{Algorithm2}, then it holds that\\
(i) The sequence $\{H_{\delta_k}(x^k)\}$ is monotonically nonincreasing. In particular, for any $k\in\mathbb{N}$, it holds 

that
$$H_{\delta_{k+1}}(x^{k+1})-H_{\delta_k}(x^k)\leq\frac{\zeta(2-\gamma)+\gamma-1}{2\gamma\eta_{\mu^{k}}}\|x^{k+1}-x^{k}\|^{2}.$$\nonumber\\
(ii)~ $\lim_{k\to\infty}\left\|{x}^{k+1}-{x}^k\right\|^2\to 0.$

\end{lemma}
\begin{proof}
(i)~To simplify the notations in our analysis, we denote
\begin{align}
H_{\delta_k}(x^k)=Q(x^k)+\delta_k\|x^k -x^{k-1}\|^2\quad\mathrm{with}\quad\delta_k:=\frac{\zeta}{2\gamma\eta_{\mu^{k}}}.
\label{eq:25}
\end{align}
In the following, we show that the sequence is monotonically nonincreasing.
By following \eqref{eq:19}, the path search procedure finds a new update ${x}^{k+1}$ to make $m_{\mu^{k}}({x}^{k+1};{y}^{k})$ an upper bound of $s({x}^{k+1})$, thus we have
\begin{align}
&r({x}^k)+\left\langle{g}_{\mu^k},{x}^k-{y}^k\right\rangle+\frac{1}{2\gamma\eta_{\mu^k}}\left\|{x}^k-{y}^k\right\|^2 \nonumber \\
&\geq r({x}^{k+1})+\left\langle{g}_{\mu^k},{x}^{k+1}-{y}^k\right\rangle+\frac{1}{2\gamma\eta_{\mu^k}}\left\|{x}^{k+1}-{y}^k\right\|^2.
\label{eq:26}
\end{align}
From equation \eqref{eq:26}, we obtain
\begin{align}
r({x}^{k+1})-r({x}^k)&\leq
\left\langle{g}_{\mu^k},{x}^k-{x}^{k+1}\right\rangle
+\frac{1}{2\gamma\eta_{\mu^k}}[\left\|{x}^k-{y}^k\right\|^2-\left\|{x}^{k+1}-{y}^k\right\|^2].
\label{eq:27}
\end{align}
\begin{align}
&s({x}^{k+1})\leq s({y}^k)+\left\langle{g}_{\mu^k},{x}^{k+1}-{y}^k\right\rangle+\frac{1}{2\eta_{\mu^k}}\left\|{x}^{k+1}-{y}^k\right\|^2.
\label{eq:28}
\end{align}
Based on the convexity of the quadratic function
$s$, we obtain that
\begin{align}
s({x}^{k})\geq s({y}^k)+\left\langle{g},{x}^{k}-{y}^k\right\rangle.
\label{eq:29}
\end{align}
Using equations \eqref{eq:28} and \eqref{eq:29}, we derive
\begin{align}
s({x}^{k+1})-s({x}^k)\leq &\left\langle{-g},{x}^k-{y}^{k}\right\rangle+\left\langle{g}_{\mu^k},{x}^{k+1}-{y}^k\right\rangle 
+\frac{1}{2\eta_{\mu^k}}\left\|{x}^{k+1}-{y}^k\right\|^2.
\label{eq:30}
\end{align}
Combining \eqref{eq:27} and \eqref{eq:30}, we derive
\begin{align}
Q(x^{k+1})-Q(x^{k})&=r(x^{k+1})-r(x^{k})+s(x^{k+1})-s(x^{k})\nonumber\\
&\leq\left\langle{-g},{x}^{k}-{y}^k\right\rangle+\left\langle{g}_{\mu^k},{x}^{k}-{y}^k\right\rangle \nonumber\\
&+(\frac{1}{2\eta_{\mu^k}}-\frac{1}{2\gamma\eta_{\mu^k}})\left\|{x}^{k+1}-{y}^k\right\|^2
+\frac{1}{2\gamma\eta_{\mu^k}}\left\|{x}^{k}-{y}^k\right\|^2.
\label{eq:31}
\end{align}
Due to $\left\langle{-g+{g}_{\mu^k}},{x}^{k}-{y}^k\right\rangle\leq 0$ in Algorithm \ref{Algorithm2}, where $g=\nabla s\left({y}^{k}\right)$ and ${g}_{\mu^k}=\frac{\langle \nabla s\left({y}^{k}\right),{d}(\mu^{k})\rangle}{\|{d}(\mu^{k})\|^{2}}{d}(\mu^{k})$, its corresponding geometric interpretation is illustrated in Figure \ref{NEfig1}.

\begin{figure}[!th]
     \centering
    \includegraphics[width=0.64\textwidth]{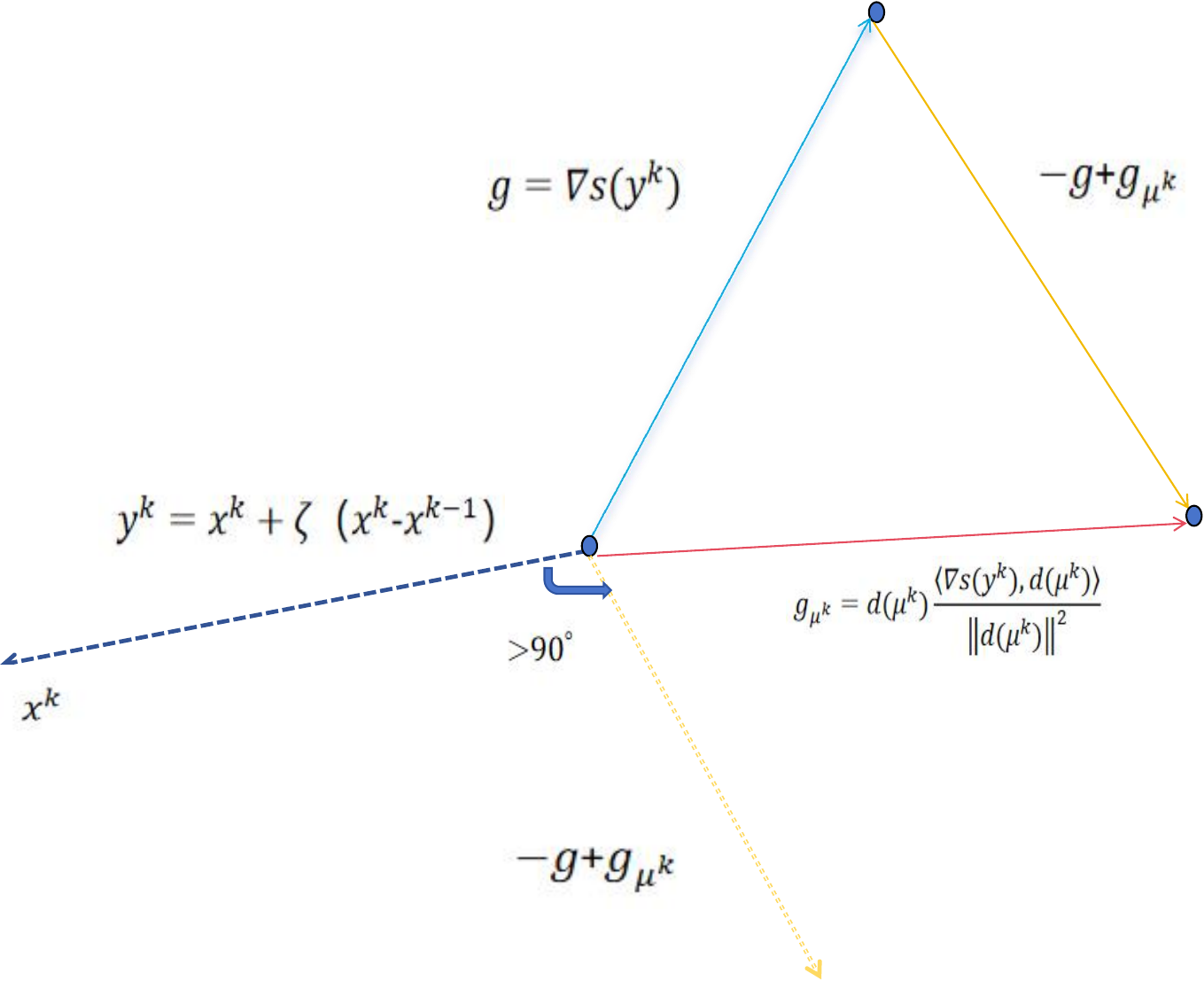}
    \caption{The geometric interpretation of the inequality $\left\langle-g+g_{\mu^k},x^k-y^k\right\rangle\leq0$.}
      \label{NEfig1}
\end{figure}
\vspace{-0.5cm}

\begin{align}
&Q(x^{k+1})-Q(x^{k}) \nonumber \\
\leq&(\frac{1}{2\eta_{\mu^k}}-\frac{1}{2\gamma\eta_{\mu^k}})\left\|{x}^{k+1}-{y}^k\right\|^2
+\frac{1}{2\gamma\eta_{\mu^k}}\left\|{x}^{k}-{y}^k\right\|^2 \nonumber \\
\leq&(\frac{1}{2\gamma\eta_{\mu^k}}-\frac{1}{2\eta_{\mu^k}})[(\zeta-1)\|x^{k+1}-x^{k}\|^{2}+\zeta(1-\zeta)\|x^{k}-x^{k-1}\|^{2}]
+\frac{\zeta^2}{2\gamma\eta_{\mu^k}}\left\|{x}^{k}-{x}^{k-1}\right\|^2 \nonumber\\
\leq&(\frac{1}{2\gamma\eta_{\mu^k}}-\frac{1}{2\eta_{\mu^k}})(\zeta-1)\|x^{k+1}-x^{k}\|^{2}
+[\zeta(1-\zeta)(\frac{1}{2\gamma\eta_{\mu^k}}-\frac{1}{2\eta_{\mu^k}})+\frac{\zeta^2}{2\gamma\eta_{\mu^k}}]\|x^{k}-x^{k-1}\|^{2}\nonumber \\
\leq&(\frac{1}{2\gamma\eta_{\mu^k}}-\frac{1}{2\eta_{\mu^k}})(\zeta-1)\|x^{k+1}-x^{k}\|^{2}
+[(\frac{\zeta}{2\gamma\eta_{\mu^k}}-\frac{\zeta}{2\eta_{\mu^k}})-\frac{\zeta^2}{2\gamma\eta_{\mu^k}}+\frac{\zeta^2}{2\eta_{\mu^k}}+\frac{\zeta^2}{2\gamma\eta_{\mu^k}}]\|x^{k}-x^{k-1}\|^{2}\nonumber \\
\leq&(\frac{1}{2\gamma\eta_{\mu^k}}-\frac{1}{2\eta_{\mu^k}})(\zeta-1)\|x^{k+1}-x^{k}\|^{2}+[(\frac{\zeta}{2\gamma\eta_{\mu^k}}-\frac{\zeta}{2\eta_{\mu^k}})+\frac{\zeta^2}{2\eta_{\mu^k}}]\|x^{k}-x^{k-1}\|^{2}\nonumber \\
\leq&(\frac{1}{2\gamma\eta_{\mu^k}}-\frac{1}{2\eta_{\mu^k}})(\zeta-1)\|x^{k+1}-x^{k}\|^{2}+\frac{\zeta}{2\gamma\eta_{\mu^k}}\|x^{k}-x^{k-1}\|^{2}.
\label{eq:33}
\end{align}
we know that
\begin{align}
\|x^{k+1}-{y}^k\|^2 &=\|x^{k+1}-x^{k}-\zeta(x^{k}-x^{k-1})\|^{2} \nonumber \\
&=\|x^{k+1}-x^k\|^2-2\zeta\langle x^{k+1}-x^k,x^k-x^{k-1}\rangle+\zeta^2\|x^k-x^{k-1}\|^2 \nonumber \\
& \geq(1-\zeta)\|x^{k+1}-x^{k}\|^{2}+\zeta(\zeta-1)\|x^{k}-x^{k-1}\|^{2},
\label{eq:34}
\end{align}
where the inequality follows from the fact that
\begin{align}
2\langle x^{k+1}-x^k,x^k-x^{k-1}\rangle\leq\|x^{k+1}-x^k\|^2+\|x^k-x^{k-1}\|^2.
\label{eq:35}
\end{align}
In the second inequality of equation \eqref{eq:33} is obtained by combining equations \eqref{eq:8}, \eqref{eq:34}, and \eqref{eq:35}. Based on the definition in equation \eqref{eq:25} and the result in equation \eqref{eq:33}, we can derive equation \eqref{eq:36}.
\begin{align}
H_{\delta_{k+1}}(x^{k+1})-H_{\delta_k}(x^k)
&=\left(Q(x^{k+1})+\frac{\zeta}{2\gamma\eta_{\mu^{k+1}}}\|\Delta_{k+1}\|^2\right)-\left(Q(x^k)+\frac{\zeta}{2\gamma\eta_{\mu^k}}\|\Delta_k\|^2\right)\nonumber \\
&\leq (\frac{\zeta-1}{2\gamma\eta_{\mu^k}}-\frac{\zeta-1}{2\eta_{\mu^k}}+\frac{\zeta}{2\gamma\eta_{\mu^{k+1}}})\|x^{k+1}-x^{k}\|^{2}\nonumber\\
&\leq \frac{\zeta(2-\gamma)+\gamma-1}{2\gamma\eta_{\mu^k}}\|x^{k+1}-x^{k}\|^{2}.
\label{eq:36}
\end{align}
Due to $0<\zeta<\frac{1-\gamma}{2-\gamma}$ in Algorithm \ref{Algorithm2},
then the sequence follows $H_{\delta_k}(x^{k})$  is monotonically nonincreasing.

(ii)~Then, summing up \eqref{eq:36} from $k=0,1,\ldots,N$  and $x^{-1}=x^0$, it yields
\begin{align}
 \sum_{k=0}^N(\frac{\zeta(\gamma-2)+1-\gamma}{2\gamma\eta_{\mu^k}})\|x^{k+1}-x^{k}\|^{2}&\leq\sum_{k=0}^N\left(H_{\delta_k}(x^k)-H_{\delta_{k+1}}(x^{k+1})\right) \nonumber\\
&=H_{\delta_0}(x^0)-H_{\delta_{N+1}}(x^{N+1}) \nonumber\\
&=Q(x^{0})-H_{\delta_{N+1}}(x^{N+1}) \nonumber\\
&\leq Q(x^{0})-Q<\infty.\nonumber\\
 \label{eq:37}
\end{align}
Given that $\{\eta_{\mu^k}\}_{k\in\mathbb{N}}$ is bounded, we derive that
\begin{align}
\lim_{k\to\infty}\left\|{x}^{k+1}-{x}^k\right\|^2\to0.
\label{eq:38}
\end{align}
Since $x^{k+1}$ is set to $v^{k+1}$ whenever $v^{k+1}$ yields a smaller objective value, the following condition must also be satisfied:
\begin{align}
\lim_{k\to\infty}\left\|{v}^{k+1}-{x}^k\right\|^2\to0.
\label{eq:39}
\end{align}
For the proof of this statement, refer to \cite{frankel2015splitting}. This concludes the proof.
\end{proof}

\begin{theorem}\label{thm:3.3} Suppose that $H_{\delta}$ is lower-bounded, $s$ is a quadratic function, and $r$ is a lower semicontinuous function. Let $\{x^k\}_{k \in N}$ be a sequence generated by the Algorithm \ref{Algorithm2} converging to $x^{\star}$. Then  $0 \in \partial H_{\delta}(x^{\star})$, i.e., $x^{\star}$ is a critical point.
\end{theorem}
\begin{proof} 
The theorem can be proved by combining Lemma \ref{lem:3.3} (which establishes step-size boundedness) with part (ii) of Lemma \ref{lem:3.7}, following the methodology of \cite[Theorem 1]{li2015accelerated}.
\end{proof}
We now analyze the global convergence rate of the PDOME using the KL property. Initially, we show that $g_{\mu^k}$ converges to the gradient direction of $s$ as $k \rightarrow \infty$.
\begin{lemma}\label{lem:8} Let $e^k={g}_{\mu^k}-\nabla s({y}^k)$. The sequence $\left\{\left\|e^k\right\|\right\}_{k\in\mathbb{N}}$ converges to 0 as $k\to\infty$ and we have
\begin{align}
\|\partial H_{\delta_{k+1}}(x^{k+1})\|
\leq(L_{s}+\frac{1+\zeta}{\gamma\eta_{\mu^k}}\left)\|{x}^{k+1}-{x}^k\right\|
+(\zeta L_{s}+\frac{\zeta}{\gamma\eta_{\mu^k}})\left\|{x}^{k}-{x}^{k-1}\right\|+\left\|{e}^k\right\|.
\label{eq:40}
\end{align}
\end{lemma}
 
\begin{proof}
By analyzing the optimality condition of \eqref{eq:19}, we obtain that\\
$$\left\|{g}_{\mu^k}+\frac{1}{\gamma\eta_{\mu^k}}\left({x}^{k+1}-{y}^k\right)-\nabla s({x}^{k+1})\right\|\in\left\|\partial Q({x}^{k+1})\right\|.
$$
By triangle inequality and smoothness of $\nabla s$, we have
\begin{align*}
\left\|\partial Q({x}^{k+1})\right\|&\leq\left\|{g}_{\mu^k}-\nabla s({x}^{k+1})\right\|+\frac{1}{\gamma\eta_{\mu^k}}\left\|{x}^{k+1}-{y}^k\right\| & \\
  & \leq\left\|\nabla s({y}^k)-\nabla s({x}^{k+1})\right\|+\frac{1}{\gamma\eta_{\mu^k}}\left\|{x}^{k+1}-{y}^k\right\|+\left\|e^k\right\| &\\
  & \leq\left\|(1+\zeta)\nabla s({x}^k)-\zeta\nabla s({x}^{k-1})-\nabla s({x}^{k+1})\right\|+\frac{1}{\gamma\eta_{\mu^k}}\left\|{x}^{k+1}-{y}^k\right\|+\left\|e^k\right\| &\\
  & \leq\left\|\nabla s({x}^k)-\nabla s({x}^{k+1})+\zeta\nabla s({x}^{k})-\zeta\nabla s({x}^{k-1})\right\|+\frac{1}{\gamma\eta_{\mu^k}}\left\|{x}^{k+1}-{y}^k\right\|+\left\|e^k\right\| &\\
  & \leq\left\|\nabla s({x}^k)-\nabla s({x}^{k+1})\|+\zeta\|\nabla s({x}^{k})-\nabla s({x}^{k-1})\right\|&\\
  &+\frac{1}{\gamma\eta_{\mu^k}}\left\|{x}^{k+1}-{x}^k\right\|+\frac{\zeta}{\gamma\eta_{\mu^k}}\left\|{x}^{k}-{x}^{k-1}\right\|+\left\|e^k\right\|&\\
  &\leq(L_{s}+\frac{1}{\gamma\eta_{\mu^k}}\left)\|{x}^{k+1}-{x}^k\right\|+(\zeta L_{s}+\frac{\zeta}{\gamma\eta_{\mu^k}})\left\|{x}^{k}-{x}^{k-1}\right\|+\left\|{e}^k\right\|.&
  \end{align*}
Then, we have
\begin{align}
\|\partial H_{\delta_{k+1}}(x^{k+1})\|&=\|\partial[Q(x^{k+1})+\frac{\zeta}{2\gamma\eta_{\mu^{k+1}}}\left\|x^{k+1}-x^{k}\right\|^2]\|\nonumber\\
&\leq(L_{s}+\frac{1+\zeta}{\gamma\eta_{\mu^k}}\left)\|{x}^{k+1}-{x}^k\right\|+(\zeta L_{s}+\frac{\zeta}{\gamma\eta_{\mu^k}})\left\|{x}^{k}-{x}^{k-1}\right\|+\left\|{e}^k\right\|.&
\label{eq:41}
\end{align}
Based on $\lim_{k\to\infty}\left\|x^{k+1}-x^k\right\|^2\to0$ and $\lim_{k\to\infty}\left\|x^{k}-x^{k-1}\right\|^2\to0$ in Theorem \ref{thm:3.3}, we have
$\lim_{k\to\infty}\left\|e^k\right\|\to0.$
\end{proof}
\begin{theorem}\label{thm:3.4} Suppose that $H_{\delta}$ satisfies the KL property on $\omega(x^k)$ which is the cluster point set of $\{x^k\}_{k\in\mathbb{N}}$, then the sequence $\{x^k\}_{k\in\mathbb{N}}$ generated by Algorithm \ref{Algorithm2} has summable residuals, $\sum_{k=0}^{\infty}\|x^{k+1}-x^k\|<\infty$.
\end{theorem}
\begin{proof}
Following the same procedure as in \cite[Theorem 2.9]{attouch2013convergence}, and considering the descent property in \eqref{eq:36} together with the property from \eqref{eq:40} that $g_{\mu^k}$ converges to the gradient direction of $s$ as $k \rightarrow \infty$. one can easily show that the sequence
$\left\{{x}^k\right\}_{k\in\mathbb{N}}$ has a finite length.
\end{proof}
Whenever the KL property is invoked, we shall adopt the results given in Theorem \ref{thm:3.4}. We consider a sequence 
$\{x^k\}_{k\in\mathbb{N}}$ in Algorithm \ref{Algorithm2}, computed by means of an abstract algorithm satisfying the following inequality:\\
$\mathbf{H}_1$ (Sufficient decrease): From equation \eqref{eq:36}, for each $k\in\mathbb{N}$, where $\frac{\zeta(\gamma-2)+1-\gamma}{2\gamma\eta_{\mu^k}}>0,$ $$H_{\delta_{k+1}}(x^{k+1})+(\frac{\zeta(\gamma-2)+1-\gamma}{2\gamma\eta_{\mu^k}})\|x^{k+1}-x^k\|^2\leq H_{\delta_{k}}(x^k).$$
$\mathbf{H}_2$ (Relative error): For each $k\in\mathbb{N}$, where $\frac{\gamma \eta_{u^{k}}}{(\Re \zeta + 1)(L_{s}\cdot\gamma\cdot\eta_{u^{k}} + 1) + \zeta}>0$ and $\varepsilon_{k+1}=\frac{\gamma \eta_{u^{k}} \|e^{k}\|}{(\Re \zeta + 1)(L_{s} \cdot\gamma \cdot\eta_{u^{k}} + 1) + \zeta}\geq0$,
$$\frac{\gamma \eta_{u^{k}}}{(\Re \zeta + 1)(L_{s}\cdot\gamma\cdot\eta_{u^{k}} + 1) + \zeta} \|\partial H_{\delta_{k+1}}(x^{k+1})\| \leq \|x^{k+1} - x^{k}\| + \frac{\gamma \eta_{u^{k}} \|e^{k}\|}{(\Re \zeta + 1)(L_{s} \cdot\gamma\cdot \eta_{u^{k}} + 1) + \zeta}.$$
Let us assume that $\|x^{k} - x^{k-1}\|\leq \Re\|x^{k+1} - x^{k}\|$ holds for some $\Re>0$.
In conjunction with formula \eqref{eq:41}, we arrive at $\mathbf{H}_2$.\\
$\mathbf{H}_3$ (Parameters): The sequences $(\frac{\zeta(\gamma-2)+1-\gamma}{2\gamma\eta_{\mu^k}})_{k\in\mathbb{N}}$, $(\frac{\gamma \eta_{u^{k-1}}}{(\Re \zeta + 1)(L_{s}\cdot\gamma\cdot\eta_{u^{k-1}} + 1) + \zeta})_{k\in\mathbb{N}}$ and $(\varepsilon_k)_{k\in\mathbb{N}}$ satisfy\\
(i) $\frac{\zeta(\gamma-2)+1-\gamma}{2\gamma\eta_{\mu^k}}\ge\underline{a}>0$ for all $k\ge0$;\\
(ii) $(\frac{\gamma \eta_{u^{k-1}}}{(\Re \zeta + 1)(L_{s}\cdot\gamma\cdot\eta_{u^{k-1}} + 1) + \zeta})_{k\in\mathbb{N}}\notin l^1$;\\
(iii)$\sup_{k\in\mathbb{N}^\star}\frac{2\eta_{{u}^k}[(\Re\zeta+1)(Ls\cdot\gamma\cdot\eta_{{u}^{k-1}}+1)+\zeta]}{[\zeta(\gamma-2)+(1-\gamma)]\eta_{{u}^{k-1}}}<+\infty$;\\
(iv) $(\varepsilon_k)_{k\in\mathbb{N}}\in l^1$.
\begin{theorem}\label{thm:3.5} Let $H_{\delta}$ have the KL property at a global minimum $x^{\star}$ of $H$. Let $\left\{{x}^k\right\}_{k\in\mathbb{N}}$ be a sequence satisfying $\mathbf{H} _{1}, \mathbf{H} _{2}$ and $\mathbf{H} _{3}$ with $\epsilon _{k}\equiv 0.$ There,  exist $\rho>0$ such that if $x^0\in\mathcal{B}_\rho(\hat{x})$, then $\left\{{x}^k\right\}_{k\in\mathbb{N}}$ has finite length and converges to a global minimum $x^{\star}$  of $H_{\delta}$.
\end{theorem}
\begin{proof}
As noted in \cite{attouch2013convergence}, Theorem \ref{thm:3.5} allows for a more general formulation. For example, when $x^{\star}$ is a local minimum of $H_{\delta}$, a growth property holds locally (refer to \cite[Remark 2.11]{attouch2013convergence}  for details).

The proof of Theorem 3.1 and Theorem 3.2 draws on the reasoning presented in \cite[Section 2.3]{attouch2013convergence}, with adaptations made to account for the existence of errors and the variability shown by the parameters.
\end{proof}

\begin{theorem}\label{thm:3.6}
Let $\{x^k\}_{k \in \mathbb{N}}$ be any sequence generated by Algorithm \ref{Algorithm2}. Suppose that $H_{\delta}$ satisfies the KL property on the cluster points of $\{x^k\}_{k \in \mathbb{N}}$ with exponent $\theta \in (0, 1)$, then $\{x^k\}_{k \in \mathbb{N}}$ converges to $x^{\star}$ such that $0 \in \partial H_{\delta}(x^{\star})$ and the following inequalities hold. Assume~$\varphi(t)=\frac{C}{\theta}t^{\theta}$ for some $C>0$, $\theta\in[0,1]$.\nonumber\\
(i)If $\theta=1$ and $\inf_{k\in\mathbb{N}}\frac{[\zeta(\gamma-2)+(1-\gamma)]\gamma\eta_{u^{k}}}{2[(\Re\zeta+1)(Ls\cdot\gamma\cdot\eta_{u^{k}}+1)+\zeta]^{2}}>0$, then $x^{k}$ converges in finite time.\nonumber\\
(ii)If $\theta\in[\frac{1}{2},1]$, $\operatorname*{sup}_{k\in\mathbb{N}}\frac{\gamma \eta_{u^{k-1}}}{(\Re \zeta + 1)(L_{s}\cdot \gamma\cdot \eta_{u^{k-1}} + 1) + \zeta} <+\infty$ and
 $\operatorname*{inf}_{k\in\mathbb{N}}\frac{\zeta(\gamma-2)+1-\gamma}{2[(\Re\zeta + 1)(L_{s}\cdot\gamma\cdot\eta_{u^{k}} + 1) + \zeta]}>0$,
there   exist $c>0$ and $k_{0}\in\mathbb{N}$ such that
\begin{itemize}
   \item $\begin{aligned}[t]
      H_{\delta}(x^{k}) - H_{\delta}(x^{\star}) &= O\left(\exp\left(-c\sum_{n=k_0}^{k-1}\frac{\gamma \eta_{u^{n}}}{(\Re \zeta + 1)(L_{s}\cdot \gamma\cdot \eta_{u^{n}} + 1) + \zeta}\right)\right),
   \end{aligned}$
   
   \item $\begin{aligned}[t]
      \|x^{\star} - x^k\| &= O\left(\exp\left(-\frac{c}{2}\sum_{n=k_0}^{k-2}\frac{\gamma \eta_{u^{n}}}{(\Re \zeta + 1)(L_{s}\cdot\gamma\cdot \eta_{u^{n}} + 1) + \zeta}\right)\right).
   \end{aligned}$
\end{itemize}
 (iii)If $\theta\in[0,\frac{1}{2}],\operatorname*{sup}_{k\in\mathbb{N}}\frac{\gamma \eta_{u^{k-1}}}{(\Re \zeta + 1)(L_{s} \cdot\gamma\cdot \eta_{u^{k-1}} + 1) + \zeta} <+\infty$ and
 $\operatorname*{inf}_{k\in\mathbb{N}}\frac{\zeta(\gamma-2)+1-\gamma}{2[(\Re\zeta + 1)(L_{s}\cdot\gamma\cdot\eta_{u^{k}} + 1) + \zeta]}>0$,
 there is $k_{0}\in\mathbb{N}$ such that
 
\begin{itemize}
   \item $\begin{aligned}[t]
      H_{\delta}(x^k)-H_{\delta}(x^{\star}) &= O\bigg(\bigg(\sum_{n=k_0}^{k-1}\frac{\gamma \eta_{u^{n}}}{(\Re \zeta + 1)(L_{s}\cdot\gamma\cdot\eta_{u^{n}} + 1) + \zeta}\bigg)^{\!\!\frac{-1}{1-2\theta}}\bigg),
   \end{aligned}$
   
   \item $\begin{aligned}[t]
      \|x^\star-x^k\| &= O\bigg(\bigg(\sum_{n=k_0}^{k-2}\frac{\gamma \eta_{u^{n}}}{(\Re \zeta + 1)(L_{s}\cdot \gamma\cdot \eta_{u^{n}} + 1) + \zeta}\bigg)^{\!\!\frac{-\theta}{1-2\theta}}\bigg).
   \end{aligned}$
\end{itemize}

\begin{proof}The proof technique follows the route in \cite{frankel2015splitting}. We present the proof detail for the case $\theta\in[0,1)$, because the relation  implies that the Algorithm \ref{Algorithm2} enjoys a linear convergence rate which differs from the local convergence rate analysis based on KL property
$\operatorname{in}$ \cite{attouch2013convergence, li2015accelerated, stella2017forward, themelis2018forward}. Let $R_k:=H_{\delta_{k}}(x^k)-H_{\delta}(x^{\star})\geq0$, we can suppose that $R_k>0$ for all $k\in\mathbb{N}$, because otherwise the algorithm terminates in a finite number of steps. Since $x^k$ converges to $x^{\star}$, there exists $k_0\in\mathbb{N}$ such that, for all $k\geq k_0$, we have $x^k\in\mathcal{B}_\rho(\hat{x})$ where the KL inequality holds. Using successively $\mathbf{H}_1,\mathbf{H}_2$ and the KL inequality, we obtain
\begin{align}
\varphi^{\prime2}(R_{k+1})(R_k-R_{k+1})
& \geq\varphi^{\prime2}(R_{k+1})\frac{[\zeta(\gamma-2)+(1-\gamma)]\gamma\eta_{u^{k}}}{2[(\Re \zeta+1)(Ls \cdot\gamma\cdot\eta_{u^{k}}+1)+\zeta]^{2}}\|\partial H_{\delta}(x^{k+1})\|^2 \nonumber\\ &\geq\frac{[\zeta(\gamma-2)+(1-\gamma)]\gamma\eta_{u^{k}}}{2[(\Re \zeta+1)(Ls \cdot\gamma\cdot\eta_{u^{k}}+1)+\zeta]^{2}}.\tag{42}\label
{eq:42}
\end{align}
 for each $k\geq k_0$. Let us now consider different cases for $\theta{:}$

 Case $\theta=1{:}$ Suppose that $R_k>0$ for all $k\in\mathbb{N}.$ Then, for each $k\geq k_0$, we have 
 \[
 C^2(R_k-R_{k+1})\geq \frac{[\zeta(\gamma-2)+(1-\gamma)]\gamma\eta_{u^{k}}}{2[(\Re \zeta+1)(Ls\cdot \gamma\cdot\eta_{u^{k}}+1)+\zeta]^{2}}\geq\inf_{k\in\mathbb{N}}\frac{[\zeta(\gamma-2)+(1-\gamma)]\gamma\eta_{u^{k}}}{2[(\Re \zeta+1)(Ls\cdot \gamma\cdot\eta_{u^{k}}+1)+\zeta]^{2}}>0.
 \]
 Since $R_k$ converges, we must have $\inf_{k\in\mathbb{N}}\frac{[\zeta(\gamma-2)+(1-\gamma)]\gamma\eta_{u^{k}}}{2[(\Re \zeta+1)(Ls\cdot \gamma\cdot\eta_{u^{k}}+1)+\zeta]^{2}}=0$, which is a contradiction. Therefore, there exists some $k\in\mathbb{N}$ such that $R_k=0$, and the algorithm terminates in a finite number of steps.

 Case $\theta\in[0,1]:$ $\nu_{k}=\frac{\gamma \eta_{u^{k-1}}}{(\Re \zeta + 1)(L_{s}\cdot \gamma \cdot\eta_{u^{k-1}} + 1) + \zeta}$, $\bar{\nu}:=\sup_{k\in\mathbb{N}}\frac{\gamma \eta_{u^{k-1}}}{(\Re \zeta + 1)(L_{s}\cdot \gamma \cdot\eta_{u^{k-1}} + 1) + \zeta}$,\\
 $\overline{m}: =\inf_{k\in\mathbb{N}}\frac{\zeta(\gamma-2)+1-\gamma}{2[(\Re\zeta + 1)(L_{s}\cdot\gamma\cdot\eta_{u^{k}} + 1) + \zeta]}$, $c=\frac{\overline{m}}{C^2(1+\bar{\nu})}$ and, for
each $k\in\mathbb{N},\beta_k:=\frac{\nu_k \overline{m}}{C^2}.$ For each $k\geq k_0$, \eqref{eq:42} gives
\begin{align}
(R_k-R_{k+1})\geq\frac{\frac{[\zeta(\gamma-2)+(1-\gamma)]\gamma\eta_{u^{k}}}{2[(\Re \zeta+1)(Ls\cdot \gamma\cdot\eta_{u^{k}}+1)+\zeta]^{2}}R_{k+1}^{2-2\theta}}{C^2}\geq\beta_{k+1}R_{k+1}^{2-2\theta}.
\tag{43}\label{eq:43}
\end{align}

Subcase $\theta\in[\frac{1}{2},1]:$ Since $R_k\to0$ and $0<2-2\theta\leq1$, we may assume, by enlarging $k_0$ if necessary, that $R_{k+1}^{2-2\theta}\geq R_{k+1}$ for all $k\geq k_0.$ Inequality \eqref{eq:43} implies $(R_k-R_{k+1})\geq\beta_{k+1}R_{k+1}$, equivalently, $R_{k+1}\leq R_k\left(\frac1{1+\beta_{k+1}}\right)$ for all $k\geq k_0.$ By induction, we obtain
$$R_{k+1}\leq R_{k_0}\left(\prod_{n=k_0}^k\frac{1}{1+\beta_{n+1}}\right)=R_{k_0}\exp\left(\sum_{n=k_0}^k\ln\left(\frac{1}{1+\beta_{n+1}}\right)\right),$$
for all $k\geq k_0$. However, $\ln\left(\frac{1}{1+\beta_{n+1}}\right)\leq\frac{-\beta_{n+1}}{1+\beta_{n+1}}\leq\frac{-1}{1+\bar{\nu}}\beta_{n+1}$, and so\\
$$R_{k+1}\leq R_{k_{0}}\exp\left\{\sum_{n=k_{0}}^{k}\left(\frac{-1}{1+\bar{\nu}}\beta_{n+1}\right)\right\}=R_{k_{0}}\exp\left(-c\sum_{n=k_{0}}^{k}b_{n+1}\right).$$

Subcase  $\theta\in[0,\frac{1}{2}]:$ Recall from inequality \eqref{eq:43} that $R_{k+1}^{2\theta-2}(R_k-R_{k+1})\geq\beta_{k+1}.$
Setting $\phi(t):=\frac{C}{1-2\theta}{t}^{2\theta-1}$, we immediately obtain $\phi^{\prime}(t)=-Ct^{2\theta-2}$, and
$$\phi(R_{k+1})-\phi(R_k)=\int_{R_k}^{R_{k+1}}\phi^{\prime}(t)dt=C\int_{R_{k+1}}^{R_k}t^{2\theta-2}dt\geq C(R_k-R_{k+1})R_k^{2\theta-2}.$$
On the one hand, if we suppose that $R_{k+1}^{2\theta-2}\leq2R_k^{2\theta-2}$, then

$$\phi\left(R_{k+1}\right)-\phi\left(R_k\right)\geq\frac{C}{2}(R_k-R_{k+1})R_{k+1}^{2\theta-2}\geq\frac{C}{2}\beta_{k+1}.$$
On the other hand, we have that $R_{k+1}^{2\theta-2}>2R_k^{2\theta-2}$. Since $2\theta-2<2\theta-1<0$, we
obtain $\frac{2\theta-1}{2\theta-2}>0$. Thus $R_{k+1}^{2\theta-1}>\Lambda R_k^{2\theta-1}$, where $\Lambda:=2^{\frac{2\theta-1}{2\theta-2}}>1.$ Therefore,
$$\phi(R_{k+1})-\phi(R_k)=\frac{C}{1-2\theta}(R_{k+1}^{2\theta-1}-R_k^{2\theta-1})>\frac{C}{1-2\theta}(\Lambda-1)R_k^{2\theta-1}\geq C^{\prime},$$
with $C^{\prime}:=\frac{C}{1-2\theta}(\Lambda-1)R_{k_0}^{2\theta-1}>0.$ Since $\beta_{k+1}\leq\frac{\bar{\nu}\overline{m}}{C^2}$, we can write
$$\phi\left(R_{k+1}\right)-\phi\left(R_k\right)\geq\frac{C^{\prime}C^2}{\bar{\nu}\overline{m}}\beta_{k+1}.$$
Setting $c:=\min\{\frac{C}{2},\frac{C^{\prime}C^2}{\bar{\nu}\overline{m}}\}>0,$ we can write $\phi(R_{k+1})-\phi(R_{k})\geq c\beta_{k+1}$
for all $k\geq k_0.$ This implies
$$\phi(R_{k+1})\geq\phi(R_{k+1})-\phi(R_{k_0})=\sum_{n=k_0}^k\phi(R_{n+1})-\phi(R_n)\geq c\sum_{n=k_0}^k\beta_{n+1},$$
which is precisely $R_{k+1}\leq D\left(\sum_{n=k_0}^{k} \frac{\gamma \eta_{u^{n}}}{(\Re \zeta + 1)(L_{s}\cdot\gamma\cdot\eta_{u^{n}} + 1) + \zeta}\right)^{\frac{-1}{1-2\theta}}$ with $D=\left(\frac{c\overline{m}(1-2\theta)}{C^3}\right)^{\frac{-1}{1-2\theta}}.$
 \end{proof}
\end{theorem}
\section{Experiments}\label{sec:numerical_experiment}
We compare sPDOME and PDOME with widely recognized first-order methods, including PG \cite{combettes2011proximal} and mAPG \cite{li2015accelerated}, and second-order methods, such as PANOC \cite{stella2017simple} and its variant \cite{de2022proximal}, as well as PDOM \cite{zhou2024proximal}. All algorithm comparisons were performed on a Windows 10 computer equipped with an Intel(R) Core(TM) i7-12700 2.10 GHz CPU and 16GB of memory, with all algorithms implemented and run in MATLAB. Each benchmark algorithm is fine-tuned for a fair comparison. The hyperparameter settings in the  Algorithm \ref{Algorithm1} are as follows: $\eta=1/L_s$, $\gamma=0.98$ and $\epsilon^{\mathrm{abs}}=\epsilon^{\mathrm{rel}}=10^{-12}$. The hyperparameter settings in the Algorithm \ref{Algorithm2} are as follows: $\eta=1/L_s$, $\gamma=0.94$ and $\epsilon^{\mathrm{abs}}=\epsilon^{\mathrm{rel}}=10^{-12}$. All algorithms share the same randomly selected initial point $x^{0}$ and are stopped if $\left\|x^{k+1}-x^{k}\right\|/\left(1+\left\|x^{k+1}\right\|\right)<10^{-8}$ or $k>2000$. In the $k$-th iteration, we calculate the subdifferential and the normalized recovery error, NRE$(k)=\|x^k-x^\star\|/\|x^\star\|.$
\subsection{Nonconvex Sparse Recovery Problem} The sparse recovery aims to recover the original sparse signal from an under-determined set of measurements.
\begin{align}
\min_{x\in\mathbb{R}^n}\frac{1}{2}\|y-{\Delta}{\Upsilon}{x}\|^2+\lambda\|{x}\|_0,
\tag{44}\label{eq:44}
\end{align}
where $\Delta\in\mathbb{R}^{m\times n}$ denotes the measurement matrix, $\Upsilon\in\mathbb{R}^{n\times n}$ is the sparse transformation basis and $\lambda>0.$ In this context, the overall sensing matrix is specifically the subsampled Discrete Cosine Transform (DCT) \cite{stankovic2018analysis}, where $\Delta$  serves as the sub-sampling matrix and $\Upsilon$ represents the DCT basis.

The Hessian of $\bm{\mathrm{M}}$ in \eqref{eq:44} takes the form of $\Upsilon^\top\Delta^\top\Delta\Upsilon.$ The fast matrix inversion involves expressing $\left(\Upsilon^{\top}\Delta^{\top}\Delta\Upsilon\right)^{-1}$ as $\Upsilon ^{- 1}\left ( \Delta ^{\top }\Delta \right ) ^{- 1}{\Upsilon^{\top}}^{- 1}= \Upsilon ^{\top }\left ( \Delta ^{\top }\Delta \right ) ^{- 1}\Upsilon$, exploiting the transpose property of the DCT basis. Notably, $\Delta^{\top}\Delta$ is a rank-deficient diagonal matrix. To ensure the invertibility, a small positive value $\iota$ is incorporated along the diagonal. Due to its diagonal structure, the computational complexity of the inversion is $\mathcal{O}(n).$
\begin{figure}[h!]
  \begin{subfigure}[b]{0.32\textwidth}
    \includegraphics[width=\textwidth]{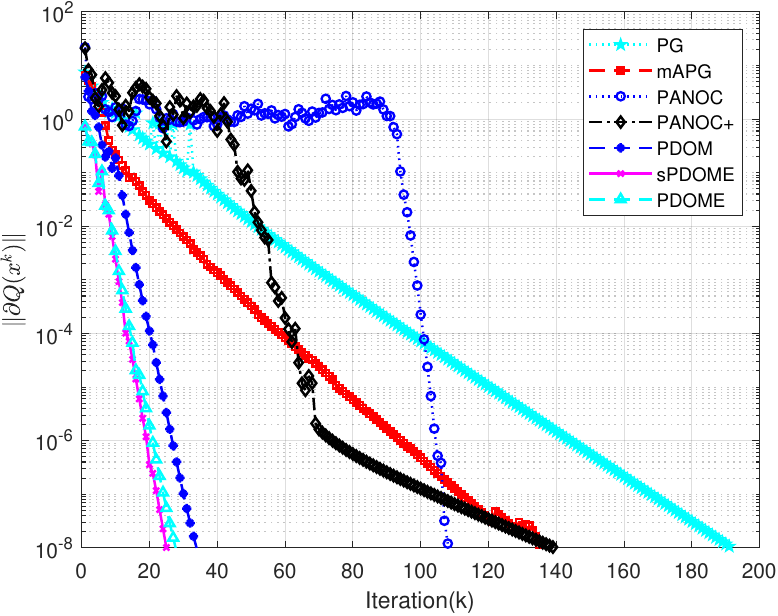}
    \caption{}
  \end{subfigure}
  \hfill
  \begin{subfigure}[b]{0.32\textwidth}
    \includegraphics[width=\textwidth]{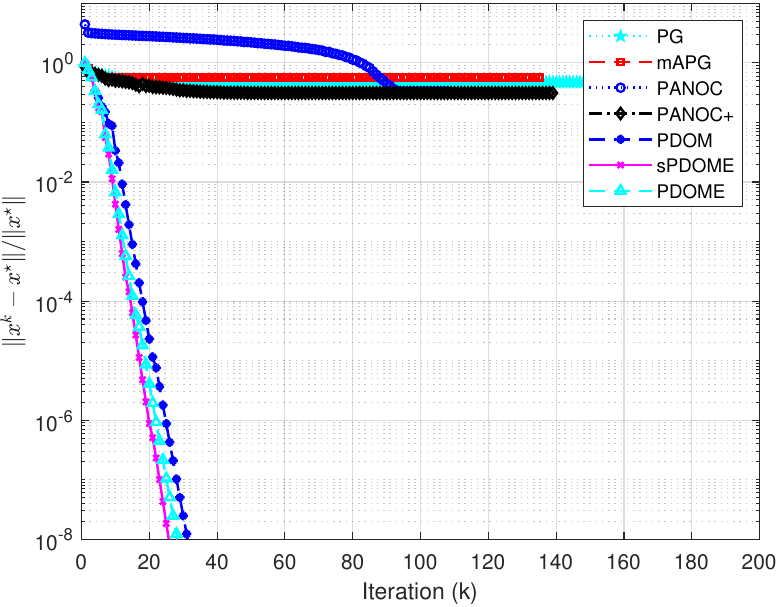}
    \caption{}
  \end{subfigure}
    \hfill
  \begin{subfigure}[b]{0.32\textwidth}
    \includegraphics[width=\textwidth]{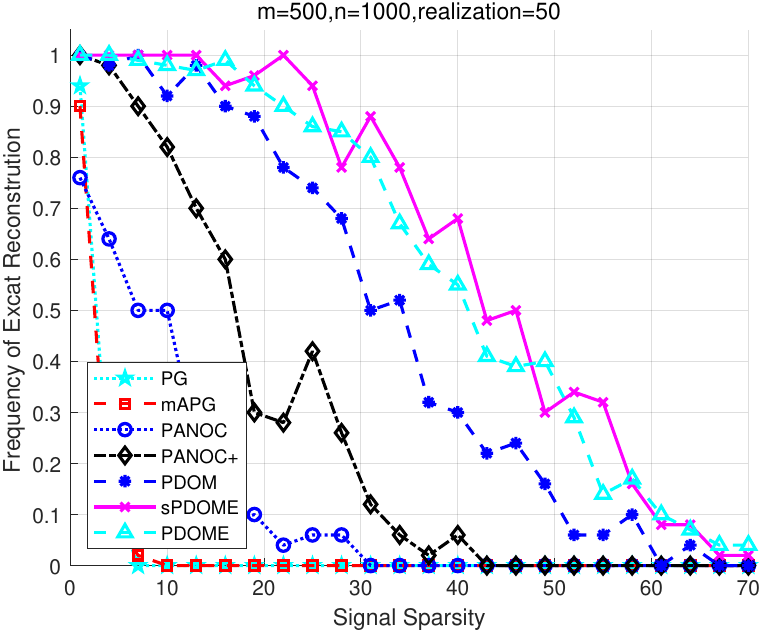}
    \caption{}
  \end{subfigure}
  \caption{Left: Performance comparisons of subdifferential. Middle: Performance comparisons of normalized recovery error of
sparse signal. Right: Phase transition curve of $\ell_0$ sparse recovery at varying sparsities.}
   \label{NEfig2}
\end{figure}
The experimental settings are summarized below: $m=$ $n/2$, ${y}={\Delta\Upsilon}x^\star$, $A=\Delta\Upsilon$, where the ground truth $x^\star$ is sparse with randomly generated entries and $\lambda=0.1\left|{A}^T{y}\right|_\infty$, following the strategy outlined in \cite{van2009probing}.

\begin{table}[htbp]
  \centering
   \caption{ Average NRE and number of iterations to reach $\|u^{k+1}\| < 10^{-12}$, $u^{k+1} \in \partial Q(x^{k+1})$ for 20 independent trials with different $m$. Sparsity level = 0.01$m$.}
  \begin{tabular}{c c c c}
    \toprule
    $m$  & 100 & 500 & 1000 \\
    Algorithm & NRE/\#Iter & NRE/\#Iter & NRE/\#Iter \\
    \midrule
    PG
 & 7.00e-01 / 140.8 & 1.82e-01 / 171.2 & 9.24e-02 / 117.8 \\
  mAPG
 & 1.00e+00 / 2.0 & 4.88e-01 / 182.1& 2.91e-01 / 134.1 \\
    PANOC
 & 3.41e+01 / 29.2 & 6.24e-01 / 38.8  & 4.56e-01 / 5.2 \\
    PANOC+
 & 8.12e-14 / 13.4  & 1.30e-02 / 36.0 & 2.29e-13 / 44.0 \\
    PDOM
 & 5.10e-13 / 151.8 & 2.40e-13 / 38.0 & 1.74e-13 / 27.2\\
    \textbf{sPDOME}
 & \textbf{8.68e-13 / 15.4} & \textbf{6.12e-13 / 18.9} & \textbf{6.88e-13 / 18.7}\\
    \textbf{PDOME}
 & \textbf{9.11e-13 / 16.7} & \textbf{7.26e-13 / 24.9} & \textbf{6.20e-13 / 27.3}\\
    \bottomrule
  \end{tabular}
 
  \label{tab:II} 
\end{table}
In Subfigures $(a)$ and $(b)$ of Figure \ref{NEfig2} depict specific convergence behavior of the compared algorithm for one realization with $m=500$ in Table \ref{tab:II}. sPDOME and PDOME outperform other baseline algorithms in terms of convergence speed and global optimality. In Subfigure $(c)$ of Figure \ref{NEfig2}, the phase transition curve is illustrated, where realizations with random initialization are considered successful if NRE$<10^{-4}.$  The results demonstrate a significantly higher success recovery rate for the sPDOME and PDOME algorithms compared to the benchmark algorithms.

\subsection{Nonconvex Sparse Approximation Problem}
Here, we address the challenge of identifying a sparse solution to a least-squares problem. As elaborated in \cite{xu2012l_}, this is accomplished by tackling the following nonconvex optimization problem:
\begin{align}
\text{minimize }\frac12\|Ax-b\|^2+\lambda\|x\|_{1/2}^{1/2},
\tag{45}\label{eq:45}
\end{align}
where $\lambda>0$ is a regularization parameter, and $\|x\|_{1/2}=\left(\sum_{i=1}^n|x_i|^{1/2}\right)^2$ is the quasi-norm $\ell_{1/2}$, a nonconvex regularizer whose role to induce the solution of \eqref{eq:45}. Function $\|x\|_{1/2}^{1/2}$ is separable, and its proximal mapping can be
\text{computed in closed form as follows, see (\cite[Theorem 1]{xu2012l_} ): for} $i=1,\ldots,n$.
$$\left[\mathrm{prox}_{\eta\left\|\cdot\right\|_{1/2}^{1/2}}(x)\right]_i=\frac{2x_i}{3}\left(1+\cos\left(\frac{2\pi}{3}-\frac{2p_\eta(x_i)}{3}\right)\right),$$
where $p_\eta ( x_i) = \textbf{arccos}\left ( \eta / 8\left ( | x_i| / 3\right ) ^{- 3/ 2}\right ).$ We performed experiments using the setting of \cite[Sec. 8.2]{daubechies2010iteratively}: matrix $A\in\mathbb{R}^{m\times n}$ has $m=n/5$ rows and was generated with random Gaussian entries, with zero mean and variance $1/m.$ Vector $b$ was generated as $b=Ax_\mathrm{orig}+v$ where $x_\mathrm{orig}\in\mathbb{R}^n$ was randomly generated with $k=5$ nonzero normally distributed entries, and $v$ is a noise vector with zero mean and variance $1/m$. Then we solved problem \eqref{eq:45} using $x^{0}=0$ as starting iterate for all algorithms. In this experiment, Figure \ref{NEfig3} shows that the proposed algorithms are effective.
\begin{figure}[!th]
    \centering
    \begin{subfigure}[b]{0.32\textwidth}
        \includegraphics[width=\linewidth]{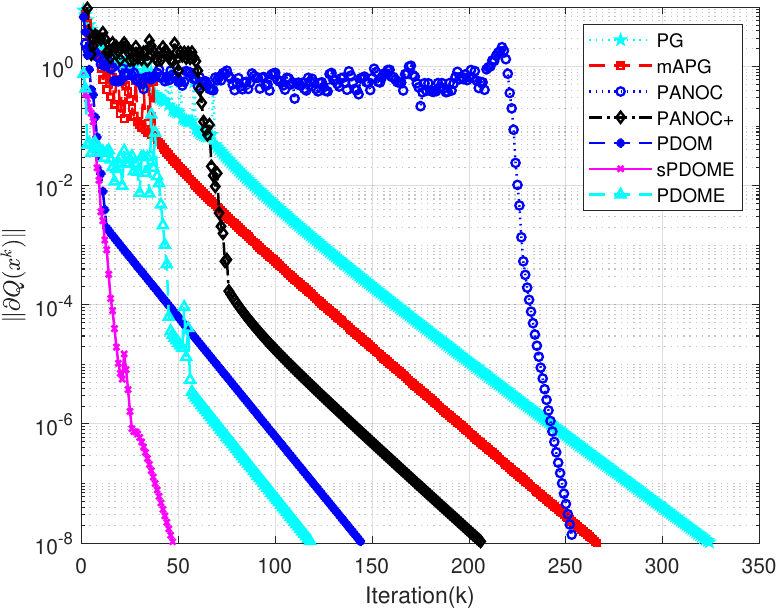}
        \caption{}
    \end{subfigure}
    \hfill
    \begin{subfigure}[b]{0.32\textwidth}
        \includegraphics[width=\linewidth]{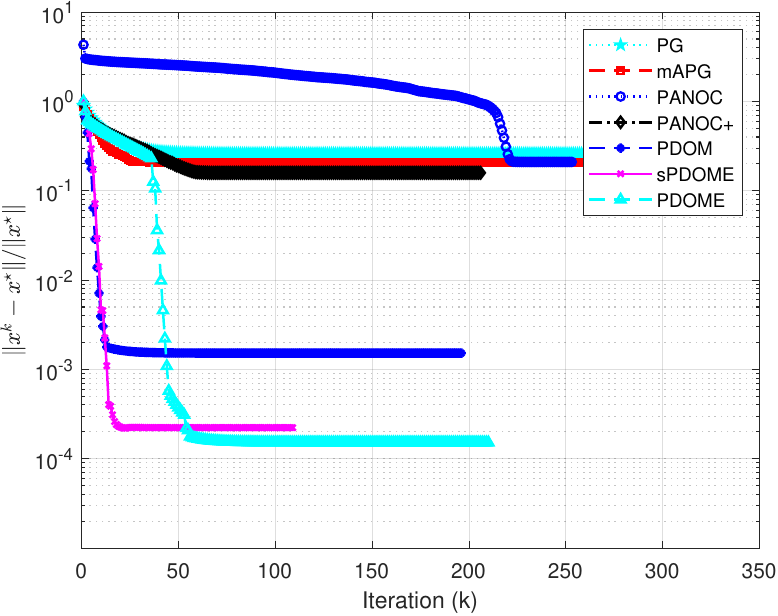}
        \caption{}
    \end{subfigure}
    \hfill
    \begin{subfigure}[b]{0.32\textwidth}
        \includegraphics[width=\linewidth]{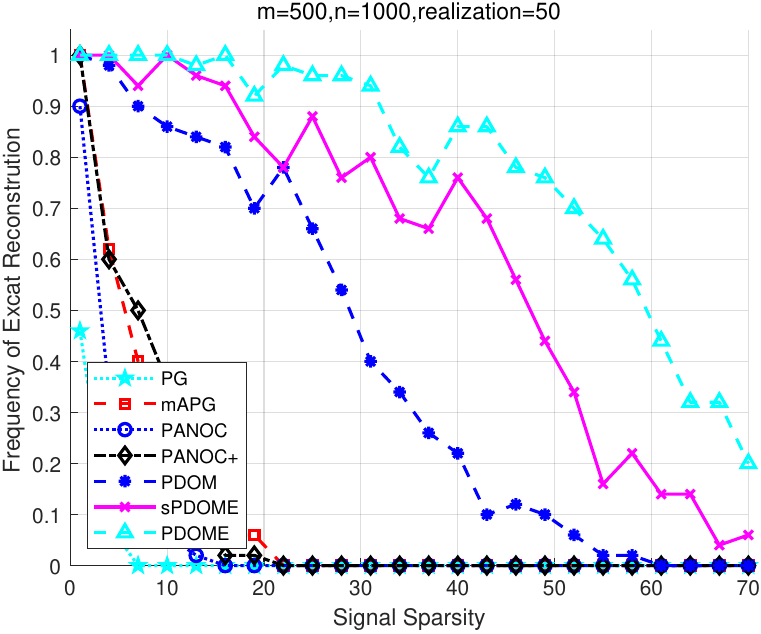}
        \caption{}
    \end{subfigure}
    \caption{Left: Performance comparisons of subdifferential. Middle: Performance comparisons of normalized recovery error of
sparse signal. Right: Phase transition curve of $\ell_{1/2}$ sparse recovery at varying sparsities.}
    \label{NEfig3}
\end{figure}

\begin{table}[htbp]
  \centering
   \caption{ Average NRE and number of iterations to reach $\|u^{k+1}\| < 10^{-12}$, $u^{k+1} \in \partial Q(x^{k+1})$ for 20 independent trials with different $m$. Sparsity level = 0.01$m$.}
  \begin{tabular}{c c c c}
    \toprule
    $m$  & 100 & 500 & 1000 \\
    Algorithm & NRE/\#Iter & NRE/\#Iter & NRE/\#Iter \\
    \midrule
 PG
\newline  & 2.36e-02 / 348.3  & 5.25e-02 / 122.7 & 5.24e-02 / 87.3 \\
 mAPG
\newline  & 3.08e-05 / 640.5& 3.92e-05 / 243.9& 4.11e-05 / 102.7 \\
PANOC
\newline  & 1.19e+00 / 971.9 & 1.17e-03 / 122.8  & 1.26e-03 / 5.0 \\
PANOC+
\newline  & 2.36e-04 / 296.5    & 3.59e-04 / 123.5 & 3.89e-04 / 56.6\\
PDOM
\newline  & 7.13e-04 / 1621.1  & 1.08e-03 / 55.8 & 1.17e-03 / 232.8\\
\textbf{sPDOME}
\newline  & \textbf{3.12e-04 / 426.1}& \textbf{1.90e-04 / 138.6} & \textbf{1.81e-04 / 47.5}\\
\textbf{PDOME}
\newline  & \textbf{1.24e-04 / 1377.1}& \textbf{1.34e-04 / 313.7} & \textbf{1.36e-04 / 97.6}\\

    \bottomrule
  \end{tabular}
  
  \label{tab:III} 
\end{table}
The first two subfigures $(a)$ and $(b)$ of Figure \ref{NEfig3}  illustrate the convergence behavior of a single realization with m=500 from Table \ref{tab:III}. Compared with the benchmark algorithms, the sPDOME and PDOME algorithms converge to the critical point faster and achieve smaller recovery errors. Subfigure $(c)$ of Figure \ref{NEfig3} shows that the successful recovery rates of the sPDOME and PDOME algorithms are significantly higher than those of the benchmark algorithms.

Table \ref{tab:III} provides an overview of the average performance of the proposed algorithm and benchmark methods across various problem sizes. It is noticeable that sPDOME and PDOME converge more rapidly and exhibit a stronger capability to reach a better optimum than other algorithms. This is clearly reflected in the fact that they require fewer iterations to get close to the critical point and achieve a smaller NER.
\section{Conclusion}\label{sec:5}
In this paper, the PDOME algorithm are proposed for nonconvex and nonsmooth problems with a quadratic term. 
During the iteration process of PDOME algorithm, constructs and minimizes a majorant function in a hybrid direction based on extrapolation. Theoretical analysis confirms that the algorithm can achieve convergence to a critical point, and its global convergence rate is studied based on the KL property. Numerical experiments show that the sPDOME and PDOME algorithms converges faster in nonconvex problems and can better converge to a local optimal solution. Building on the research in this paper, future work will further expand the content related to quadratic functions and Bregman distances at the algorithm level.
\section*{Declarations}
\noindent{\bf Funding:} This research is supported by the National Natural Science Foundation of China (NSFC) grants 92473208, 12401415, the Key Program of National Natural Science of China 12331011, the 111 Project (No. D23017),  the Natural Science Foundation of Hunan Province (No. 2025JJ60009), the Tianchi Talent Program of Xinjiang Uygur Autonomous Region (CZ001328).

\noindent{\bf Data Availability:} Enquiries about data/code availability should be directed to the authors.

\noindent{\bf Competing interests:} The authors have no competing interests to declare that are relevant to the content of this paper.

\bibliographystyle{elsarticle-num}	

\bibliography{references.bib}

\end{document}